\documentclass[11pt, twoside]{article}
\usepackage[a4paper, width=150mm,top=25mm,bottom=25mm]{geometry}
\usepackage{fancyhdr}
\usepackage{titlesec}

\titleformat*{\section}{\bfseries \center}
\titleformat*{\subsection}{\bfseries \center}

\usepackage{float}

\newcommand{\light}[1]{\textcolor{gray}{#1}}
\newcommand{\ignore}[1]{}

\usepackage{amsmath}
\usepackage{amssymb}
\usepackage{amsthm}
\usepackage{mathtools}
\usepackage{IEEEtrantools}
\usepackage{tikz}
\usepackage[colorlinks = true,
            linkcolor = blue,
            urlcolor  = red,
            citecolor = blue,
            anchorcolor = blue]{hyperref}
\usepackage{caption}

\usetikzlibrary{calc}
\usepackage{hhline}
\usepackage{tabularx}
\usepackage{caption}

\usepackage{etoolbox}
\patchcmd{\thebibliography}{\section*{\refname}}{}{}{}

\usepackage{hyperref}

\usetikzlibrary{arrows} 
\usetikzlibrary{decorations.markings}
\usetikzlibrary{svg.path}

\newtheorem{theorem}{Theorem}

\newtheorem{corollary}[theorem]{Corollary}
\newtheorem{prop}[theorem]{Proposition}

\newtheorem{definition}[theorem]{Definition}

\DeclarePairedDelimiter\abs{\lvert}{\rvert}

\DeclarePairedDelimiter\ip{\langle}{\rangle}

\begin{document}

\renewenvironment{abstract}
 {\small
  \begin{center}
  \bfseries \abstractname\vspace{-.5em}\vspace{0pt}
  \end{center}
  \list{}{%
    \setlength{\leftmargin}{10mm}% <---------- CHANGE HERE
    \setlength{\rightmargin}{\leftmargin}%
  }%
  \item\relax}
 {\endlist}

\title{\textbf{Computation of Scattering Matrices and their Derivatives for Waveguides}}

\author{Greg Roddick}
\date{28 November 2020; $3^{rd}$ Revision}

\maketitle

\begin{abstract}
This paper describes a new method to calculate the stationary scattering matrix and its derivatives for Euclidean waveguides. This is an adaptation and extension to a procedure developed by Levitin and Strohmaier which was used to compute the stationary scattering matrix on surfaces with hyperbolic cusps~\cite{alexnew}, but limited to those surfaces. At the time of writing, these procedures are the first and only means to  explicitly compute such objects. In this context the challenge we faced was that on Euclidean waveguides, the scattering matrix naturally inhabits a Riemann surface with a countably infinite number of sheets making it more complicated to define and compute. We overcame this by breaking up the waveguide into compact and non-compact components, systematically describing the resolvent for the Neumann Laplace operator on both of them, giving a thorough treatment of the Riemann surface, and then using a ``gluing" construction~\cite{rm} to define the resolvent on the whole surface. From the resolvent, we were able to obtain the scattering matrix. The algorithm we have developed to do this not only computes the scattering matrix itself on such domains, but also arbitrarily high derivatives of it directly. We have applied this, together with the finite element method, to calculate resonances for a selection of domains and will present the results of some numerical calculations in the final section. Whilst this is certainly not the first, nor only method to compute resonances on these domains, ie. Levitin and Marletta have done so previously \cite{lm} and Aslanyan, Parnovski and Valiliev before them~\cite{apv} and other techniques, such as perfectly matched layers may be adapted for this purpose~\cite{Jiang_2020}. The method described here has several advantages in terms of speed and accuracy and moreover, provides more information about the scattering phenomena. 
\end{abstract}

\begin{keywords}
scattering matrix, scattering theory, waveguide, resonances
\end{keywords}

\tableofcontents

\section{Introduction}
\subsection{Background overview and outline of the paper}
Over the course of this paper, a method to calculate the stationary scattering matrix, and its derivatives for Euclidean waveguides with cylindrical ends will be presented, along with the results from some numerical calculations making use of it. 
\\
\\
We will proceed as follows: For the remainder of this section, we will give the reader an overview of the background material and introduce the notion of waveguides and scattering matrices. The main body of the paper will introduce and describe each concept necessary for the end result and ether provide the reader with proofs or references where appropriate. Section~\ref{sheetsstuff} will be devoted to describing the domain of the resolvent function for the Neumann Laplace operator on $\mathbb{R}_+,$ the set of real numbers with positive values. This will be used in section~\ref{geftosm} to define and describe the concept of generalised eigenfunctions and the scattering matrix itself, and show how the latter may be obtained from the former. Section~\ref{ndmsection} will focus on the Neumann to Dirichlet map, both its definition and calculation, and its use for computing the scattering matrix. We will then go on to explain how to compute the scattering or $S$ matrix in section~\ref{smcalcsection}, culminating in an explicit formula for it in equation~(\ref{finalsm}). The method from section~\ref{smcalcsection} will be extended in section~\ref{derrivssection} to calculate derivatives. With a fast and efficient method of computing the $S$ matrix and its derivatives, we apply it and produce numerical results. The background to the numerical computations will be presented in section~\ref{resbackground} and tabulated data will be presented in section~\ref{numresults} along with links to some animations.
\\
\\
Waveguides are piece-wise, path connected subsets of $\mathbb{R}^n$ that can be written as the union of a compact domain and any number of non-compact, cylindrical ends. The compact and non-compact parts share a common boundary (see Figure~\ref{overflow}). Whilst the ends are non-compact and infinite, they have a certain regularity that allows explicit calculation on them, the compact part may have a more complicated geometry which precludes this. The stationary scattering matrix, in this context, describes the outcome of a scattering event; the scattering of a wave packet in the compact domain, originating at infinity. In any such event, a proportion of the incoming wave packet will be transmitted and a proportion reflected; the coefficients of the scattering matrix contain this information. 

The ends  of the waveguide can be thought of as the Cartesian product of their boundary with the positive, real half-line $\mathbb{R}_+$. A notable feature of Euclidean waveguides is that the scattering matrix admits a meromorphic continuation to a certain Riemann surface with a countably infinite number of sheets~\cite{christiensen}. We will present the reader with a proof of this. The existence of a meromorphic continuation is a property inherited from the resolvent of the Laplace operator on these domains whose integral kernel is made up a direct sum involving complex square roots. In order to construct this meromorphic continuation, one must first construct a meromorphic continuation of the resolvent. To do this, we will use a well-known ``gluing" construction described in detail by Melrose (see \cite{rm}) which we have adapted for waveguides. The construction makes use of the meromorphic Fredholm theorem and the fact that the resolvent, for the Neumann Laplace operator on the ends of the waveguide, can be easily computed as an integral kernel. The resolvent can then be used to construct generalised eigenfunctions, a generalisation of the notion of eigenfunctions, and from them the scattering matrix. This is the first time such techniques have been explicitly adapted to prove this and, at the time of writing, they remain the only such method.
\\
\\
The Neumann to Dirichlet map, takes boundary data for solutions to partial differential equations with Neumann boundary conditions and maps it to the corresponding data for solutions to the same equations with Dirichlet boundary conditions. It is a vital component of this algorithm and we make heavy use of the quick and flexible method developed by Levitin and Marletta~\cite{lm} to compute it. They were able to formulate the Neumann to Dirichlet map in terms of an infinite sum of Dirichlet data of Neumann eigenvalues on the compact, part of the waveguide. The advantage of this method is that the computation of the Dirichlet data of Neumann eigenvalues, the most computationally costly step, typically using the Finite Element Method, has to be performed only once. Though there are other methods to numerically solve partial differential equations, the point stands. Previously, if one wished to compute the Neumann to Dirichlet map, multiple numerical solutions of the underlying pde. would be required each time (See Definition~\ref{ndmapbasis} and Section~\ref{lmtechnique} which follows it for the contrasting techniques). We will present both methods and we will also show that the scattering matrix, as well as its derivatives, can be obtained from this data directly. Levitin and Strohmaier have already used this technique to obtain the scattering matrix on finite volume, non-compact hyperbolic surfaces~\cite{alexnew}. Due to the more complicated nature of the Riemann surface on which the resolvent sits, the  problem of determining the scattering matrix for waveguides, which we provide a solution to here, is more complex.
\\
\\
Being in possession of a fast and efficient algorithm to compute the scattering matrix and its derivatives enables the calculation of resonances which we define to be poles of the scattering matrix. We are able to do this using a combination of numerical contour integration and Newton's method. The time delay and scattering length~\cite{wigner}\cite{eisenbud} can also be computed from the scattering matrix.
\\
\\
Evans, Levitin and Vassiliev's 1994 paper~\cite{trappedmodes} proved the existence of embedded eigenvalues (eigenvalues for an operator embedded within its continuous spectrum) for the Neumann Laplacian on two dimensional waveguides with an obstacle and/or deformation of the waveguide, so long as the domain has cross-sectional symmetry. 
This was further generalised  to waveguides with cylindrical ends by Davies and Parnovski~\cite{trappedmodes2}. Parnovski and Levitin have, amongst others, produced two other papers on this topic~\cite{hawkparnovski}~\cite{pencilev}. Embedded eigenvalues can be calculated numerically with their method and we have included some examples in the results section. The main focus for our numerical experiments, however, has been on complex resonances.
\\
\\
Levitin and Marletta \cite{lm} and Aslanyan, Parnovski and Vassiliev~\cite{apv}, were able to compute complex resonances for a collection of waveguides. In Levitin and Marletta's case, they computed embedded eigenvalues for a domain with cross-sectional symmetry and observed them decaying to complex resonances when a small perturbation of the domain destroyed that symmetry. We have been able to replicate their results for the same domains with our method, making a slight  improvement on accuracy. We have also performed some additional numerical calculations of our own, the results of which are presented in section~\ref{numresults}. This is certainly not the only method for computing resonances on Waveguides, perfectly matched llayers are a feasible alternative~\cite{Jiang_2020}, though to replicate he results we provide, existing work will need adaptation.
\\

\subsection{Waveguides}
\label{wgdef}
Let us denote the waveguide by $M$; it is embedded in $n$-dimensional Euclidean space with $K$ cylindrical ends. $M$ can be written as $M=E\cup X,$ where $X$ is a compact, piece-wise connected manifold, with a piece-wise smooth Lipschitz boundary. We can write the ends (the non-compact part of $M$) $E,$ whose $k$ connected components will be denoted as $E_k,$ as follows:
$$E=\Gamma \times \mathbb{R}_+=\bigcup_{k=1}^K [\Gamma_k\times \mathbb{R}_+]= \bigcup_{k=1}^K E_k.$$
Where each $\Gamma_k \subset \mathbb{R}^{n-1}$ is a compact and connected domain, with smooth boundary, and for any $i\neq j$, 
$$[\Gamma_i \times \mathbb{R}_+ ] \cap  \left[\Gamma_j \times  \mathbb{R}_+  \right] = \emptyset.$$ 
Define
$$E_k=\Gamma_k \times \mathbb{R}_+ \quad \text{    and    }\quad  \Gamma=E\cap X=\lbrace0 \rbrace \times \Gamma= \bigcup_{k=1}^K\left[ \lbrace 0 \rbrace \times \Gamma_k\right].$$
We will call the boundary of the whole waveguide $M, \ \Sigma.$ 

\begin{figure}[H]
\centering
\includegraphics[width=90mm, height=80mm]{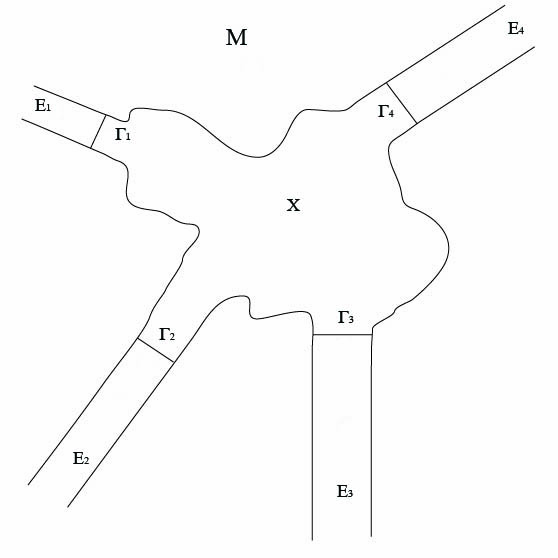}
\caption{Waveguide}
\label{overflow}
\end{figure}

On the ends $E$ we can use separation of variables to deal with the compact $\Gamma$ and non-compact $\mathbb{R}_+$ individually. Reed and Simon have given a rigorous definition of this in their text which we will reproduce. The original, with proof, can be found in here~\cite[p.~52,II.10]{rs1}: 

\begin{theorem}
\label{tens}
Let $\langle M_1 , \mu_1 \rangle$ and $\langle M_2 , \mu_2 \rangle$ be measure spaces so that  $L^2(M_1 , d \mu_1)$ and $L^2(M_2 , d \mu_2)$ are separable.
\\
\\
Then there is a unique isomorphism $\cong,$ from $L^2(M_1 , d \mu_1) \hat{\otimes} L^2(M_2 , d \mu_2)$ to \\$ L^2(M_1 \times M_2, d \mu_1 \otimes d \mu_2)$, where the $ \hat{\otimes}$, denotes the completion of the tensor product, so that  $f \otimes g \mapsto fg.$
\end{theorem}

\begin{corollary}
\label{sov}
With the ends of the waveguide $E$ defined as above, we can now write
$$L^2(E)\cong L^2(\mathbb{R}_+) \hat{\otimes} L^2(\Gamma).$$
\end{corollary}

$\Delta_{E},$ the Laplace operator on $L^2(E)$ can now be written in the form
\begin{equation}
\label{sepvar1}
\Delta_{E}= -\frac{\partial^2}{\partial x^2} \otimes  1 - 1 \otimes \Delta_{\Gamma},
\end{equation}
where $\Delta_{\Gamma}$ is the Laplace operator acting on $L^2(\Gamma), \ 1 \otimes \Delta_{\Gamma}$ is the tensor product of $\Delta_{\Gamma}$ with the identity operator on $L^2(\mathbb{R}_+)$ and $-\frac{\partial^2}{\partial x^2} \otimes  1$ is the tensor product of the differential operator $-\frac{\partial^2}{\partial x^2}$ acting on $L^2(\mathbb{R_+})$ with the identity operator on $L^2(\Gamma).$  The compactness of $\Gamma$ means that for any $\lambda$ in the resolvent set of $\Delta_\Gamma$, $L^2(\Gamma)$ has an orthonormal basis, consisting of eigenfunctions of $(\Delta_{\Gamma}-\lambda)^{-1}$. An application of the discrete Fourier transform yields:
$$L^2(\Gamma)\cong \bigoplus_{j=0}^\infty \mathbb{C} = l^2.$$

We conclude that 
\begin{equation}
\label{sepvar2}
L^2(E) \cong L^2(\mathbb{R}_+) \otimes L^2(\Gamma) \cong L^2(\mathbb{R}_+) \otimes l^2 \cong \bigoplus_{j=0}^{\infty}L^2(\mathbb{R}_+). 
\end{equation}
$(\Delta_{E_j}-\lambda)$ acts on each direct summand in~(\ref{sepvar2}) by
$-\frac{\partial^2}{\partial x^2} -\lambda+\mu_j,$ and the $\mu_j$ are the, not necessarily distinct, Neumann eigenvalues of $\Delta_\Gamma,$ enumerated in ascending order and repeated with multiplicity taken into account. We shall henceforth refer to each of the summands in Equation~\ref{sepvar2} as \textbf{modes}.
$\Delta_{E_j}-\lambda$ acts on each mode as multiplication by $\xi^2-\lambda+\mu_j$ in the Fourier spectral representation (after an application of the Fourier transform). 
\\
\\
The Schwartz Kernel Theorem proves existence and uniqueness of an integral kernel for the Neumann resolvent on a half-line and comes about after a simple calculation. See~\cite[Page 22]{tetgen} for an explicit calculation. If we take the domain be $\mathbb{R}_+$ with Neumann boundary conditions at $0,$ the kernel of the resolvent will be of the form
\begin{equation}
\label{intkernel}
\frac{-e^{i|x+y|
\sqrt{\lambda}
}}{4 i\sqrt{\lambda}}+\frac{-e^{i |x-y|
\sqrt{\lambda}}}{4 i \sqrt{\lambda}}.
\end{equation}
%\end{example}

We should take note of the fact that when $x \neq y$ the kernel is holomorphic when defined as a function of $\lambda$ in the branch of the square root where $\mathrm{Im}(\sqrt{\lambda})>0$. The anti-diagonal doesn't cause problems, because $x$ and $y$ are non-negative. Due to the ellipticity of $(\Delta-\lambda),$ elliptic boundary regularity can be invoked to show that the resolvent kernel is smooth away from the diagonal also.
\\
\\
The resolvent for $\Delta_{E},$ written as it is in (\ref{sepvar1}), can now be written as

\begin{equation}
\label{resdecompo}
R_0(\lambda)= \bigoplus_j r^j(\lambda),
\end{equation}

where $r^j(\lambda)=\frac{1}{\xi^2 + \mu_j- \lambda}$ in the spectral representation. This means that equation (\ref{intkernel}) can be slightly modified to give the kernel on each mode as

\begin{equation}
\label{neumresintker}
r^j(\lambda)=
\frac{-e^{i|x+y|
\sqrt{\lambda-\mu_j}
}}{4 i\sqrt{\lambda- \mu_j}}+\frac{-e^{i |x-y|
\sqrt{\lambda- \mu_j}}}{4 i \sqrt{\lambda- \mu_j}}.
\end{equation}
The existence of square roots, and their branches, adds extra complexity to this resolvent kernel. Rather than talking about $R_0(\lambda)$ as being defined on $\mathbb{C}$, we must instead talk about it being defined on a Riemann surface $Z,$ on which it is single valued function of $\lambda$ (see \cite{christiensen}).
\section{A description of the domain of our resolvent}
\label{sheetsstuff}
We will denote the the Neumann eigenvalues of $\Delta_\Gamma$, enumerated in ascending order with multiplicity taken into account, by $\lbrace \mu_j \rbrace$ and the set of such eugenvalues discarding repeated entries by $\lbrace \eta_j \rbrace.$ This follows a convention of notation established by Christiensen \\ \cite[Page5]{christiensen}. The reason for this is that for each $j\in \mathbb{N}$ corresponds to a branch point for $\sqrt{\lambda-\eta_j}.$ A complete description of this domain has already appeared in the paper by Guillop{\'e}~\cite{french}, we will, nevertheless, present a brief overview here.
\\
\\
In section \ref{wgdef}, we saw that the resolvent for $\Delta_E$ is made up of the direct sum of the $r^j(\lambda),$ acting on the direct sum of $L^2$ spaces; equation (\ref{neumresintker}). The Riemann surface for each individual summand that makes up this resolvent will have a branch at each of the $\eta_j.$ We define the \textbf{physical sheet} of $Z$ to be the sheet of the surface, which can be identified with $\mathbb{C}\setminus \mathbb{R}_+,$ for which all the $\sqrt{\lambda-\eta_j},$ have positive imaginary part and identify it with $\mathbb{C}\setminus \mathbb{R}_+$. The whole surface $Z$ is made up of a countable number of ``sheets" of this nature, each of which represents a choice as to whether each $\sqrt{\lambda-\eta_j}$ has a strictly positive imaginary part or not.
\\
\\
 The resolvent of $\Delta_E$ is analytic on the physical sheet for values of $\lambda$ where it is indeed the resolvent operator and not just a continuation of it. When we wish to extend the resolvent from the physical sheet to other sheets, we must do so along a path. It is necessary to have some kind of a coherent system to categorise such paths.
\\
\\
\begin{figure}[H]
  \centering
    
  \tikzset{middlearrow/.style={
        decoration={markings,
            mark= at position 0.5 with {\arrow{#1}} ,
        },
        postaction={decorate}
    }}

  \begin{tikzpicture}

    \coordinate (Origin)   at (0,0);
    \coordinate (examplepoint)   at (7,-1);
    \coordinate (XAxisMin) at (-3,0);
    \coordinate (XAxisMax) at (8,0);
    \coordinate (YAxisMin) at (0,-2);
    \coordinate (YAxisMax) at (0,5);
    \draw [very thick, black,-latex] (Origin) -- (XAxisMax);% Draw x axis
   \draw [thin, gray,-latex] (Origin) -- (YAxisMax);% Draw y axis
    \node[draw,circle,inner sep=2pt,fill, label= below :$\mu_1$] at (0,0)  {};
    \node[draw,circle,inner sep=2pt,fill, label= below :$\mu_2$] at (1.5,0) (ev2) {} ;
    \node[draw,circle,inner sep=2pt,fill, label= below :$\mu_3$] at (3.5,0) (ev3) {};
    \node[draw,circle,inner sep=2pt,fill, label= below :$\mu_4$] at (5,0) (ev4) {};
      \node[draw,circle,inner sep=2pt,fill, label= below :$\mu_5$] at (7.2,0) (ev5) {};
      \node[draw,circle,inner sep=1pt,fill] at (-1,1.5)  {};
  \node[draw,circle,inner sep=1pt,fill] at (7,-1)  {};

    \coordinate (basepoint) at (-1,1.5);

\draw [middlearrow={latex}] plot [smooth,tension=1]  coordinates {(basepoint)(0,1) (2,-0.8) (4,2) (examplepoint)};      
\draw [middlearrow={latex}] plot [smooth,tension=0.6] coordinates {(basepoint) (0.5,2.5) (5,2) (examplepoint)};
\draw [middlearrow={latex}]   plot [smooth,tension=0.6]  coordinates {(basepoint)  (0,0.6) (1.5,-1.5) (4,-1.6) (examplepoint)};

    \end{tikzpicture}
  \caption{Inequivalent paths that lift to paths with endpoints in sheets $(1,2,3,4),$ $(1,3,4),$ and $(1),$ of $Z$ respectively.}
  \label{Windingdemo}
\end{figure}
 A path in $Z$ will remain a path in $\mathbb{C}\setminus \lbrace \eta_j \rbrace$ under the covering projection. Similarly, a path in $\mathbb{C}\setminus \lbrace \eta_j \rbrace,$ lifts to a path in $Z$ if the location (in terms of which sheet) of the pre-image of one of the endpoints is known (or given). We note that paths crossing of the intervals $(\eta_j,\eta_{j+1})$ on the real line in $\mathbb{C}\setminus \lbrace \eta_j \rbrace$ correspond with crossings of the boundaries between sheets in $Z$. This means that homotopy equivalent paths in $Z$ originating in the physical sheet, paths in $\mathbb{C}\setminus \lbrace \eta_j \rbrace,$ and sheets of $Z$ are all in one-to-one correspondence with each other. Each class of paths in $\mathbb{C}\setminus \lbrace \eta_j \rbrace$ can be indexed by a finite subset of $\mathbb{N}$ that we call $ J$ constructed by counting the number of times the path crosses the $n^{th}$ interval mod $2.$ Equivalently $J$ may be defined as follows:
\begin{equation}
\label{jay}
 J= \lbrace  j\in\mathbb{N}:\mathrm{Im} (\sqrt{\lambda-\mu_j}) \leq0 \rbrace.
\end{equation}
 The monodromy theorem can be used to show that meromorphic continuations along paths are unique and hence any meromorphic continuation from the  physical to a non-physical sheet is unique no matter which path is taken. Melrose and others have already
demonstrated the existence of analytic continuations for the individual summands $r_j(\lambda)$ that make up the resolvent kernel for $\Delta_E$~\cite[Page 11]{rm}. The ``gluing" constructions and the meromorphic Fredholm theorem are used to prove the existence of a meromorphic continuation of the resolvent to $\Delta$ on the whole waveguide $M$ (this can be found in more detail in Melrose's text and Guilope's paper \cite{french}\cite{rm}).  This has been adapted by us for use on Euclidean waveguides with minimal modifications, the full proof of the existence and uniqueness of a meromorpic continuation of the resolvent on $M$ can be found in the author's thesis~\cite[Page 30]{tetgen}.
\section{Generalised Eigenfunctions and the Scattering Matrix}
\label{geftosm} 
Let $\chi$ be a function on the waveguide $M$ with support on $E$ equal to $1$ outside a compact set, and fix an orthonormal basis of Neumann eigenfunctions of  $\Delta_\Gamma,$ namely $\lbrace\nu_j(y) \rbrace$. 
As described in section~\ref{sheetsstuff}, we may identify this sheet with $\mathbb{C} \setminus  \mathrm{min}_{j\in J}[\mu_j ,\infty)$. We will generally be working with either the physical sheet of $Z,$ or the sheet defined by $J$ which we shall refer to as the non-physical sheet from now on. When identified in this way, every $\lambda$ in the non-physical sheet of $Z$ has its counterpart in the physical sheet which is identical as a complex number, but not as on the Riemann surface.\\
\\
\\
We can now define: 
\begin{IEEEeqnarray*}{ll}
\IEEEyesnumber
\label{gef}
\varphi_J(\lambda,x,y)&=\sum_{j\in \mathbb{N}} \chi e^{-i \sqrt{\lambda -\mu_j} x}\nu_j(y) - R(\lambda) \left[(\Delta  - \lambda)\left( \chi e^{-i \sqrt{\lambda-\mu_j} x}\nu_j(y) \right)\right].
\end{IEEEeqnarray*}

for $\lambda$ in the physical sheet, and for $\lambda$ in the non-physical sheet:
\begin{IEEEeqnarray*}{ll}
\varphi_J(\lambda,x,y)&= \sum_{j\notin J} \chi e^{-i \sqrt{\lambda -\mu_j} x}\nu_j(y) - R(\lambda) \left[(\Delta  - \lambda)\left( \chi e^{-i \sqrt{\lambda-\mu_j} x}\nu_j(y) \right)\right]\\
&
\IEEEyesnumber \label{gefnonps}
+ \sum_{j \in J} \chi e^{+i \sqrt{\lambda -\mu_j} x}\nu_j(y) - R(\lambda) \left[(\Delta  - \lambda)\left( \chi e^{+i \sqrt{\lambda-\mu_j} x}\nu_j(y) \right)\right].
\end{IEEEeqnarray*}

The formulae in Equations~(\ref{gef}) and (\ref{gefnonps}) describe a $\textbf{Generalised Eigenfunction}.$ A generalised eigenfunction, we define to be a solution of $(\Delta - \lambda)\varphi_J(\lambda,x,y)=0$ which is not square integrable, unlike an eigenfunction which is. A reader may observe the $R(\lambda) (\Delta  - \lambda)$ in the right hand side of the equations above and mistakinly conclude that these two operators are mutually inverse and so equations (\ref{gef}) and (\ref{gefnonps}) sum to zero. This is not the case, the reason is that $R(\lambda)$, unless $\lambda$ is on the physical sheet, is a meromorphic continuation of the resolvent and not the resolvent itself, and when $\lambda$ is in the physical sheet, the resolvent is only an inverse of $(\Delta-\lambda)$ for $L^2,$ square integrable functions. These functions have a number of properties:
 
\begin{prop}

\label{3props}
\begin{enumerate}

\item $\varphi_J(\lambda,x,y)$ is a meromorphic function of $\lambda$ for any $\lambda \in Z$ and holomorphic if $\lambda$ is in the physical sheet.  
\item For $j \in J$ and $\lambda$ in the physical sheet of $Z;$
 \begin{equation}
 \label{gef2}
 R(\lambda) \left[(\Delta  - \lambda)\left( \chi e^{-i \sqrt{\lambda-\mu_j} x}\nu_j(y) \right)\right] \in L^2 (\mathbb{R}_+).
 \end{equation}

\item There exists a unique, meromorphic $S_{J,j,k}(\lambda)$ such that on $E,$ and with $\lambda$ in the physical sheet;
\begin{IEEEeqnarray*}{lll}
\IEEEyesnumber \label{gefa}
\varphi_J(\lambda,x,y) &=&\sum_{j\in J} \left( e^{-i \sqrt{\lambda-\mu_j} x} \nu_j(y)+\sum_{k\in J} S_{J,j,k}(\lambda)e^{i \sqrt{\lambda-\mu_k} x}\nu_k(y) \right) \\
&+&\sum_{j \notin J} T_j(\lambda)e^{i\sqrt{\lambda-\mu_j}x}\nu_j(y)
\end{IEEEeqnarray*}
and
\begin{IEEEeqnarray*}{lll}
\IEEEyesnumber\varphi_J(\lambda,x,y) &=&\sum_{j\in J} \left( e^{i \sqrt{\lambda-\mu_j} x} \nu_j(y)+\sum_{k\in J} S_{J,j,k}(\lambda)e^{-i \sqrt{\lambda-\mu_k} x}\nu_k(y) \right) \\
&+&\sum_{j \notin J} T_j(\lambda)e^{i\sqrt{\lambda-\mu_j}x}\nu_j(y),
\end{IEEEeqnarray*}
for $\lambda$ in the non-physical sheet. In the case where $\lambda$ is in the physical sheet, the $S_{J,j,k}(\lambda)$ are holomorphic.
\end{enumerate}
\end{prop}

\begin{proof}
$\\$
\begin{enumerate}
\item Follows from the meromorphicity of the various functions whose products make up $\varphi_J(\lambda,x,y).$
\item This is due to the square integrability of the resolvent kernels in the physical sheet.
\item Observe that for any $j\in J$, each summand in equation (\ref{gef}) becomes zero when acted upon by $\Delta-\lambda.$ Using a simple separation of variables on $E$ we see that any solution to $(\Delta - \lambda)F(\lambda,x,y)=0,$ including the one we have, will be of the form,
\begin{equation}
\label{gefb}
\sum_{k=1}^\infty  \left(A_{j,k} (\lambda) e^{-i\sqrt{ \lambda -\mu_k } x} + B_{j,k} (\lambda) e^{+i \sqrt{ \lambda -\mu_k } x} \right)\nu_k(y).
\end{equation}

We can see that in part 2) of the proposition, the requirement that when we subtract $\chi e^{-i\sqrt{ \lambda-\mu_j} x} \nu_j(y)$ the result be square integrable, means that  $A_{j,k}=\delta_{j,k}.$
 \\
 \\
% We know that for all the $j\in J$ where $ \mathrm{Im} (i\sqrt{\lambda-\mu_j})<0,$ we have
 In order to reconcile equation (\ref{gefb}) with equation (\ref{gef}), we rename the $B_{j,k}(\lambda), \  S_{J,j,k}(\lambda)$ for $k\in J$. For $k\notin J$ ande see that the remaining terms;
 \begin{equation}
 \label{T}
B_{j,k} (\lambda) e^{+i \sqrt{ \lambda -\mu_k } x} 
 \end{equation}
are all square integrable as $k\notin J \Rightarrow \mathrm{Im}(\sqrt{\lambda-\mu_k})>0.$ 
\\
\\
When we combine the all such summands to get $\varphi(\lambda,x,y),$ as described in equation (\ref{gef}), we get
$$\sum_{j\in J} \left( e^{-i \sqrt{\lambda-\mu_j} x}\nu_j(y)+\sum_{k\in J} S_{J,j,k}(\lambda)e^{i \sqrt{\lambda-\mu_k} x}\nu_k(y) \right) +\sum_{k \notin  J} \sum_{j\in J}B_{j,k}(\lambda)e^{i\sqrt{\lambda-\mu_k}x}\nu_k(y), 
$$

where we can define, for $j \notin J,$  $T_j(\lambda)=\sum_{k\in J}B_{k,j}(\lambda) $ to finish. The result for the non-physical sheet is due to meromorphic continuation.

\end{enumerate}
\end{proof}

Those used to dealing with the dynamic scattering matrix defined, for example, in Reed-Simon~\cite{rs3}, should be aware that the two definitions can be proved to be equivalent. A fundamental property of the generalised eigenfuctions and scattering matrix are their uniqueness in the $L^2$ norm. This means that $S_J(\lambda)$ is also uniquely determined by the geometry of $M,$ the choice of basis for the $J$ Neumann eigenfunctions and $J\subset \mathbb{N}$ itself.
\\
For any given $\lambda$ in the non-physical sheet of $Z,$ given by $J$ and identified with a suitable subset of the complex plane,  we shall denote its counterpart in the physical sheet by $\lambda^*$. 

\begin{theorem}
$S_J(\lambda^*)=\overline{S_J^{-1}(\lambda)}$
\label{funceq}
\end{theorem}

\begin{proof}
Observe that $$\ip{\Delta \varphi_j(\lambda,x,y),\varphi_J(\overline{\lambda},x,y)}-\ip{\varphi_J,\Delta \varphi_J}=(\overline{\lambda}-\overline{\lambda})\ip{\varphi_J,\varphi_J}=0.$$
The right hand side of the equation above, namely the $(\overline{\lambda}-\overline{\lambda})$ is not an issue, as $\lambda$ is simply a number in this context and the $\overline{\lambda},$ its complex conjugate, will be in the same sheet of $Z$ as $\lambda$.
\\
\\
Green's second identity can now be invoked to give that
$$\int_\Gamma \left( \frac{\partial \varphi_J(\lambda^*,x,y)}{\partial n} \overline{\varphi_J(\overline{\lambda},x,y)} - \varphi_J(\lambda^*,x,y) \frac{\partial \overline{\varphi_J(\lambda,x,y)}}{\partial n} \right)=0.$$

In particular, this means that

\begin{IEEEeqnarray*}{lLl} 
&&\sum_{j,k\in J}i \sqrt{ \lambda^*-\mu_k}(\delta_{k,j}-S_{J,j,k}(\lambda^*) )  \sum_{l,m\in J} \left( \delta_{l,m}+ \overline{S_{J,m,l}(\overline{\lambda})} \right) \int_\Gamma \nu_k(y) \nu_l(y)\\
&-&\sum_{j,k\in J}(\delta_{k,j}+S_{J,j,k}(\lambda^*) )  \sum_{l,m\in J}-  i \sqrt{\lambda-\mu_l}\left( \overline{S_{J,m,l}(\lambda)}- \delta_{l,m} \right) \int_\Gamma \nu_k(y) \nu_l(y)\\
=&&\sum_{j,k\in J}i \sqrt{ \lambda^*-\mu_k}(\delta_{k,j}-S_{J,j,k}(\lambda^*) )  \sum_{l,m\in J} \left( \delta_{l,m}+ \overline{S_{J,m,l}(\overline{\lambda})} \right) \int_\Gamma \nu_k(y) \nu_l(y)\\
&-&\sum_{j,k\in J}(\delta_{k,j}+S_{J,j,k}(\lambda^*) )  \sum_{l,m\in J}  i \sqrt{\lambda^*-\mu_l}\left( \overline{S_{J,m,l}}(\lambda)- \delta_{l,m} \right) \int_\Gamma \nu_k(y) \nu_l(y)\\
=&&0
\end{IEEEeqnarray*}

From this, we obtain
$$\sum_{k,j ,m\in J} \left( (S_{J,j,k}(\lambda^*) -\delta_{k,j})  ( \delta_{k,m}+ \overline{S_{J,m,k}(\overline{\lambda})} ) +  (\delta_{k,j}+S_{J,j,k}(\lambda^*) ) ( \overline{S_{J,m,k}(\overline{\lambda})}- \delta_{k,m} ) \right)=0$$

This is due to the orthonormality of the $\lbrace \nu_j \rbrace$ and the way in which we have identified our sheet of $Z$ with a $\mathbb{C}.$ Multiplying these out and simplifying shows that, for fixed $j,m \in J,$ gives
$$\sum_{k\in J}  S_{J,j,k}(\lambda^*)  \overline{S_{J,m,k}(\overline{\lambda})}=\delta_{j,m}.$$
Since $S_J(\lambda)$ and $S_J(\overline{\lambda})$ are, by construction, the same when $\lambda$ and $\overline{\lambda}$ are in the same sheet of $Z$ the result follows.

\end{proof}

\section{The Neumann to Dirichlet map}
\label{ndmsection}
Given a solution to a pde with Neumann boundary conditions, the Neumann to Dirichlet map takes the Dirichlet boundary data and maps it to the corresponding Dirichlet boundary conditions for this solution. The Neumann to Dirichlet map will be a vital intermediate step between the resolvents, which we have extensively covered in previous sections, and the scattering matrix. On the internal domain, or compact part of the waveguide $X$, the Neumann to Dirichlet map can be numerically computed either directly using the Finite Element Method (\ref{directndm}) or indirectly using a combination of the Finite Element Method and Levitin Marletta's technique (see \ref{lmtechnique}). It is an elementary calculation on the ``ends" $E$. To give a rigorous definition, let $X$ be a Lipschitz domain in $\mathbb{R}^n, \ n \geq 2.$ Define the map 
\begin{equation}
\label{rigndm}
N:H^{-1/2}(\partial X)\rightarrow H^{1/2}(\partial X)
\end{equation}
 acting on $g\in H^{-1/2}(\partial X)$ by $$\mathcal{D}g=\varphi_{|_{\partial X}},$$
 where $\varphi$ is the solution to the Neumann problem, with $g$ as the boundary derivative. This is the inverse of the Dirichlet to Neumann map
 $$\mathcal{D}:H^{1/2}(X)\rightarrow H^{-1/2}(X),$$
whose action on $f\in  H^{1/2}$ is $\mathcal{D}f=\frac{\partial}{\partial n}(HF).$ $H$ is an extension of $f$ to a solution of $(\Delta-\lambda)(Hf)=0$ on $X.$ \cite{ip2}
\subsection{Calculating the Neumann to Dirichlet map in the internal domain}\label{ndmapcalc}
\label{directndm}
It is known that the Neumann eigenfunctions of $(\Delta-\lambda)$ on $\Gamma$, the common boundary between $X$ and $E,$ form an orthonormal basis of $L^2(\Gamma)$, with Neumann eigenvalues  \\ $\mu_j, \ j\in \mathbb{N}$~\cite{mt}. Given a basis of $L^2(\Gamma)$, we may compute the Neumann to Dirichlet map in matrix form. This is the first step towards viewing the Neumann to Dirichlet map as a concrete, computable object. With the correspondence between the Neumann to Dirichlet map, the scattering matrix and the resolvent, computing the Neumann to Dirichlet map in this way allows us to realise these other objects in a similar manner.

\begin{definition}
\label{ndmapbasis}
\textbf{Neumann to Dirichlet map, associated to $(\Delta-\lambda)$ on a basis}

Let us consider an ordered orthonormal basis of $L^2(\Gamma)$, $\lbrace \nu_j \rbrace_{j=0}^\infty$, and $\varphi_k$, such that 
\begin{equation}
\label{prob}
(\Delta - \lambda)\varphi_k=0, \quad \frac{\partial \varphi}{\partial n} |_{\Sigma}=0 ,   \quad \frac{\partial \varphi_k}{\partial n}|_{\Gamma} = \nu_k.
\end{equation}
$\Sigma$ here denotes the boundary of $M, \ \Gamma$ the common boundary by which $X$ is jointed to $E.$ Then the $k,l^{th}$ element of the Neumann to Dirichlet map, in matrix form, with respect to basis  $\lbrace \nu_k \rbrace$, will be given by
$$\langle \varphi|_\Gamma, \nu_l \rangle _{L^2(\Gamma)}.$$
\end{definition}

Obviously, when doing this calculation practically, we must truncate after a finite number of entries.

\subsection{The Levitin-Marletta method for indirect calculation of the Neumann to Dirichlet map}
\label{lmtechnique}
This technique was devised by Levitin and Marletta in~\cite{lm}. The following formula for the $k,l^{th}$ entry of $N(\lambda)$ acting on a basis of the Neumann subspace of $L^2(\Gamma), \ \lbrace \phi_i \rbrace_{i=1}^\infty$ is taken from their paper where it is derived.
\\
\begin{equation}
\label{alternate}
N_{k,l}(\lambda)=\sum_{m=1}^\infty \frac{1}{\lambda-\mu_m}\ip{\phi_k ,U_m|_\Gamma} _{L^2(\Gamma)}.\langle U_m|_\Gamma, \phi_l\rangle _{L^2(\Gamma)}.
\end{equation}

$U_m$ and $\mu_m$ are the eigenfunctions and eigenvalues of the homogeneous Neumann problem on $X$, namely the solutions to;
$$(\Delta-\mu_m)U_m=0, \ \ \frac{\partial U_m}{\partial n}
\vline_{\partial X}=0.$$
\\
\\
Equation (\ref{alternate}) gives us a method  to compute $N_{k,l}$. Truncations of both (\ref{levstrick}) and (\ref{alternate}) can be computed using the finite element method, but (\ref{alternate}) is a significant improvement over direct calculation, as once we have obtained the eigenvalues and Fourier coefficients of their associated eigenfunctions, computing $N(\lambda),$ for any $\lambda$ we wish, now only involves matrix multiplication and not numerical solutions of partial differential equations. The latter is considerably slower, more computationally costly. Levitin and Marletta presented a simple trick or method to further improve the rate of convergence, or accuracy given a fixed number of eigenvalues and eigenfunctions; the derivation of this can also be found in their paper. Where $\overline{\lambda}$ is fixed and $N_{k,l}(\tilde{\lambda})$ a directly computed Neumann to Dirichlet map we get:

\begin{equation}
\label{levstrick}
N_{k,l}(\lambda)=N_{k,l}(\tilde{\lambda})+\sum_{m=1}^\infty \frac{\tilde{\lambda}-\lambda}{\mu_m^2-\lambda\mu_m -\tilde{\lambda} \mu_m +\lambda \tilde{\lambda}}\langle\nu_k ,U_m|_\Gamma \rangle_\Gamma .\langle U_m|_\Gamma, \nu_l\rangle _\Gamma.
\end{equation}
This now gives, quadratic, as opposed to linear convergence. This process can be repeated as many times as one desires to further increase the speed of convergence. In practice, to do this once is sufficient and any repetitions of this process would greatly complicate our extension of this algorithm used for calculating derivatives of the scattering matrix.

\section{Calculating the scattering matrix}
\label{smcalcsection}
\subsection{From the Neumann to Dirichlet map to the $S$ matrix}
In~\ref{lmtechnique} we described a procedure to calculate the Neumann to Dirichlet map from the the Dirichlet data of Neumann eigenfunctions on $X$. Here we will describe a procedure to extract the the scattering matrix from such data, culminating in a formula for the $S$ matrix in (see (\ref{finalsm})). As usual, we fix a sheet of $Z$ and define $J,$ the indexing set, to be the $j\in \mathbb{N}$ such that $\mathrm{Im}(\sqrt{\lambda-\mu_j})<0.$ We can proceed as follows:
\\
\\
Fix a basis $\lbrace \nu_j(y) \rbrace$ for the space of Neumann eigenfunctions on $\Gamma$, corresponding to Neumann eigenvalues $\mu_{j}$. Proposition~\ref{3props}, tels us that that the generalised eigenfunction on the cylindrical ends $E$ of $M$, with homogeneous Neumann boundary conditions on the boundary will be of the form: 
$$\varphi_J(\lambda,x,y)= \sum_{j\in J} \left( e^{-i \sqrt{\lambda-\mu_j} x}\nu_j(y)+\sum_{k\in J} S_{J,j,k}(\lambda)e^{i \sqrt{\lambda-\mu_k} x}\nu_k(y) \right) +\sum_{j \in \mathbb{N} \backslash J} T_j(\lambda)e^{i\sqrt{\lambda-\mu_j}x}\nu_j(y), $$
\\
\\
At $x=0,$ $\varphi_J(\lambda,x,y)$ simplifies to

$$\varphi_J(\lambda,0,y)= \sum_{j\in J} \left( \nu_j(y)+\sum_{k\in J} S_{J,j,k}(\lambda)\nu_k(y) \right) + \sum_{j \in \mathbb{N} \backslash J} T_j(\lambda)\nu_j(y). $$

At $x=0,$ the normal derivative with respect to $x$ will be

$$ \sum_{j\in J} \left(  i \sqrt{ \lambda- \mu_{j}} \nu_j(y) (\lambda)-\sum_{k\in J} i \sqrt{ \lambda -\mu_k} S_{J,j,k}(\lambda)\nu_k(y) \right)-\sum_{j \in \mathbb{N} \backslash J} i \sqrt{ \lambda -\mu_j} T_j(\lambda)\nu_j(y).$$

We ought to now be able to see that, on the external domain or ``ends'' $E$, the Neumann to Dirichlet map will be the inverse matrix of the following:

\begin{definition}
\label{intndm}
$$\tilde{N}(\lambda)=  \begin{pmatrix}i\sqrt{\lambda-\mu_1}  &\cdots & \cdots &  \cdots& \\
  \vdots & i \sqrt{\lambda-\mu_{2}} &\cdots& \cdots  \\
    \vdots & \vdots &i\sqrt{\lambda-\mu_{3}}   &\cdots \\
  \vdots & \vdots  &\vdots&\ddots   \\
   \end{pmatrix}$$
\\
\end{definition}

Unfortunately $\tilde{N}(\lambda)^{-1}$ is not quite yet the Neumann to Dirichlet map for $(\Delta-\lambda)$ on $E,$ as such a map only acts on the boundary data of square integrable functions (see equation (\ref{rigndm})), thus we must project out the non square integrable modes of any such generalised eigenfunction beforehand.

\begin{definition}\label{projection}Define
$$P:L^2(\Gamma)\longrightarrow L^2(\Gamma),$$
to be the projection whose kernel spanned by the Neumann eigenfunctions of $\Gamma$ associated to $\mu_j$ for $j\in J.$
\end{definition}

Now observe that the internal domain $X$ and the external domain (or ends) $E$ share a common boundary $\Gamma$. This means that, for any generalised eigenfunction defined on $M$, the action of the Neumann to Dirichlet map calculated on $X$ composed with $P$ and the Neumann to Dirichlet map calculated on $E$ should coincide.\\
\\
Set
\begin{equation}
\label{kereq}
L=\left( PN- P\tilde{N}^{-1} \right).
\end{equation}
The matrix $L$, will have a null-space of dimension $\abs{J}$. Each element of $\mathrm{Ker}(L)$ can be equated with a one of the $\abs{J}$ summands of $\varphi(\lambda,0,y),$ as defined in (\ref{gef}) in terms of some, possibly unknown, orthonormal basis of Neumann eigenfunctions of $\Delta_\Gamma$. 
\\
\\
Applying a singular value decomposition algorithm, or some other procedure to find the kernel of a matrix, i.e. $QR$, gives $\abs{J}$ kernel vectors.
\\
\\
Let $W$ the null-space of $L$
$$ W=\lbrace w^1, \cdots , w^J \rbrace.$$
For each $\omega_j \in W$ we have a representation of a generalised eigenfunction, evaluated at $0$, of the form:

\begin{equation}
\label{gefsum}
\omega^j= \bigoplus_{j \in J} \left(\delta_{j,k}+S_{j,k}(\lambda)\right) \bigoplus_{j\in \mathbb{N} \backslash J} T_j (\lambda).
\end{equation}

We could now, in theory, extract the scattering matrix from this, but before we are able to do such things, and for our scattering matrix to be of any use to us, we need to control the basis of Neumann eigenfunctions. A numerical algorithm for singular value decomposition will not necessarily give us $S_J(\lambda)$ in terms of the basis we want; the basis of Neumann eigenfunction of $\Gamma$ that was carefully chosen when we began the calculation of $N(\lambda)$ in~\ref{directndm}. 
\\
\\
Having found a basis for the null space of  (\ref{kereq}), we restrict our attention to the elements of these vectors that represent the $J$ Fourier modes and discard the rest by means of application of the operator $(1-P).$ The image of $W$ under both $(1-P)$ and $(1-P)N(\lambda)$ forms a basis in $\mathbb{R}^J$.
\\

The linear map $\tau :\mathbb{R}^J \longrightarrow \mathbb{R}^J$, defined on the $(1-P)w^j,$ by

 $$(1-P)\omega^j \mapsto (1-P)N\omega^j ,$$
 
can now be thought of as the identity map from the basis $\lbrace (1-P)\omega^1 , \cdots ,(1-P)\omega^J \rbrace$ of $\mathbb{R}^J$, to basis $\lbrace (1-P)N\omega^1 , \cdots ,(1-P)N\omega^J \rbrace$ of $\mathbb{R}^J.$
 \\
 \\
$\tau$ must be rewritten in terms of the standard basis, whose elements represent the chosen basis of the Neumann eigenfunctions of $\Gamma$. Thus, when acting our chosen basis of $L^2(\Gamma), \ \tau$ can be written as,
\begin{equation}
\label{taudef}
\tau(\lambda)=\lbrace(1-P)N \omega^1,\cdots ,(1-P)N\omega^J\rbrace^{-1}\lbrace(1-P)\omega^1,\cdots ,(1-P)\omega^J\rbrace.
\end{equation}

Now we note that applying $N$ to each $\omega_j$ gives:

\begin{equation}
\label{twodirectsums}
N(\lambda)\omega^j=   \bigoplus_{j \in J}\frac{S_{J,j,k}(\lambda)-\delta_{j,k}}{i \sqrt{ \lambda -\mu_j}}  \bigoplus_{j\in \mathbb{N} \backslash J} \frac{T_j (\lambda)}{i\sqrt{ \lambda -\mu_j}}.
\end{equation}
 Let us define $D(\lambda,k), \ k\in \mathbb{N}$ to be the $\abs{J} \times \abs{J}$ matrix acting on the set of $J$ modes for which $\mathrm{Im}(\sqrt{\lambda-\mu_j})<0.$ For the calculation of $S_J(\lambda),$ we will only be using this with $k=0.$ 
\begin{definition}
\label{dmatrix}
$$D(\lambda,k)= \frac{\partial^k}{\partial \lambda^k} \begin{pmatrix}i\sqrt{\lambda-\mu_1}  &\cdots & \cdots &  \cdots&\cdots \\
  \vdots & i \sqrt{\lambda-\mu_{2}} &\cdots& \cdots & \cdots \\
    \vdots & \vdots &i\sqrt{\lambda-\mu_{3}}   &\cdots & \cdots\\
  \vdots & \vdots  &\vdots&\ddots & \cdots  \\
  \vdots &\vdots &\vdots  &\vdots&i \sqrt{\lambda-\mu_{\abs{J}}}.
   \end{pmatrix}$$
\end{definition}
One can see now that, as $(1-P)$ will project out the second direct sum in~(\ref{twodirectsums}), the map $(1-P)N$ will be a $J\times J$ matrix defined as follows: 
 \begin{equation}
 \label{tee}
 \tau(\lambda)=  (D(\lambda,0)+S_J(\lambda).D(\lambda,0)).(  \mathrm{Id}+ S_J(\lambda))^{-1}.
 \end{equation}

This means that finally
\begin{equation}
\label{finalsm}
S_J(\lambda)=(\tau(\lambda)-D(\lambda,0))^{-1}(-D(\lambda,0)-\tau(\lambda)).
\end{equation}

\subsection{Some words on computational cost}
When calculating the $S$ matrix in practice, there are several variables governing computational cost.
\begin{itemize}
\item The number of Neumann eigenvalues of the internal domain taken. This is akin to truncating in series in equation (\ref{levstrick}) at $m$.
\item The choice of sheet in $Z,$ namely $\abs{J}:$ In the simplest case, this is simply the number of `ends' of the domain, though it can be greatly higher.
\item The number of modes chosen $n$. This is akin to the dimension of the matrix $N(\lambda)$ where, equation (\ref{levstrick}) defines it for the $k,l^{th}$ element.
\end{itemize}

The computational cost of the of computing the $S$ matrix will depend on our choice of $m$ and $n$ as they are truncations of an infinite series. In~\ref{evaccuracy}, where the optimum number of modes and eigenvalues to use for the most accurate result is discussed; it was decided that $20$ modes and $2000$ eigenvalues was sufficient, though, in all cases $m>>n.$ This implies that the computation of $N(\lambda)$ (see~\ref{levstrick}), which involves a the multiplication of an $n \times m$ matrix with an $m \times m$ diagonal matrix and then multiplication again with an $m \times n$ matrix, will dominate any calculation off computational cost and thus the asymptotic complexity of the whole operation is that of rectangular matrix multiplication. This was classically treated as $O(m^2n),$ but reducing this continues to be an area of active research (for example, see~\cite{hormander3}).

\section{Derivatives of the S matrix}
\label{derrivssection}
This section will show that an extension to the above method can be used to calculate $S_J^{(n)}(\lambda)=\frac{\partial^n}{\partial \lambda^n}S_J(\lambda)$. We will guide the reader through this process which culminates with an explicit formula in (\ref{finalderiv}). This is an  interesting endeavor in its own right but, our numerical calculations in section~\ref{numresults} make use of the argument principle and Newton's method (see Proposition~\ref{polesint}), both of which require the first derivative of $S.$ It could, in theory, also be used to write the scattering matrix as a power series. We will show that the accuracy of the computation is as accurate as the computation for the original $S$ matrix and the asymptotic computational cost is no greater. This compares favorably with the, much cruder, alternative of picking a small $h>0$ and dividing the difference of $S(\lambda+h)$ and $S(\lambda)$ by it. We don't have to worry about the choice of $h,$ the potential for magnification of errors or the possibility of strange behaviors when $\lambda$ is a non-trivial fraction of $h$ away from a resonance.

\subsection{A Neumann to Dirichlet map for the system on the external domain}
Fix a sheet of $Z.$ If $\varphi_J(\lambda,x,y)$ is a generalised eigenfunction, then $\\(\Delta-\lambda)\varphi_J(\lambda,x,y)=0.$ Thus, when we differentiate with respect to $\lambda$, we get.
$$\frac{\partial}{\partial \lambda}(\Delta-\lambda)\varphi_J(\lambda,x,y)=(\Delta-\lambda)\varphi_J'(\lambda,x,y)-\varphi_J(\lambda,x,y)=0,$$
For any $n$ we get
$$\frac{\partial^n}{\partial \lambda^n}(\Delta-\lambda)\varphi_J(\lambda,x,y)=(\Delta-\lambda)\varphi_J^{(n)}(\lambda,x,y)-\varphi_J^{(n-1)}(\lambda,x,y)=0, $$
where $\varphi_J^{(n)}(\lambda,x,y)$ denotes $\frac{\partial^n}{\partial \lambda^n}\varphi_J(\lambda,x,y)$ for brevity. One can simply look for a solution to the resulting system of equations in a similar manner to (\ref{prob}):

\begin{IEEEeqnarray*}{lLl} 
\IEEEyesnumber \label{firstsystem}
(\Delta - \lambda)\varphi_J(\lambda,x,y)=0,   & \frac{\partial \varphi_J}{\partial n} |_{\Sigma}=0 , \\
(\Delta-\lambda)\varphi_J'(\lambda,x,y)-\varphi_J(\lambda,x,y)=0,& \frac{\partial\varphi_J'}{\partial n} |_{\Sigma}=0 ,\\
\quad \quad \quad \quad \quad \quad  \vdots&\ \ \vdots\\
(\Delta-\lambda)\varphi_J^{(n)}(\lambda,x,y)-\varphi_J^{(n-1)}(\lambda,x,y)=0,\quad& \frac{\partial\varphi_J^{(n)}}{\partial n} |_{\Sigma}=0 , .
\end{IEEEeqnarray*}
On the other hand, since $\varphi(\lambda,x,y)$ is known on $E,$ we can recall equation (\ref{gefa}) and see
$$\varphi_J(\lambda,x,y) =\sum_{j\in J} \left( e^{-i \sqrt{\lambda-\mu_j} x}\nu_j(y)+\sum_{k\in J} S_{J,j,k}(\lambda)e^{i \sqrt{\lambda-\mu_k} x}\nu_k(y) \right) +\sum_{j \notin  J} T_j(\lambda)e^{i\sqrt{\lambda-\mu_j}x}\nu_j(y), $$
and deduce that

\begin{IEEEeqnarray*}{rLl} 
\varphi_J'(\lambda,x,y) &=&\sum_{j\in J} \left( \frac{-ix}{2\sqrt{\lambda-\mu_j} }e^{-i \sqrt{\lambda-\mu_j} x}\nu_j(y)+\sum_{k\in J}\left[ S'_{J,j,k}(\lambda) - \frac{x  S_{J,j,k}(\lambda)}{2i\sqrt{\lambda-\mu_k}}\right]e^{i \sqrt{\lambda-\mu_k} x}\nu_k(y) \right)
\\ 
\IEEEyesnumber
\label{geflam}
&&+\sum_{j\notin J}\left[ T'_j(\lambda) +\frac{ixT_j(\lambda)}{2\sqrt{\lambda-\mu_j}} \right]e^{i\sqrt{\lambda-\mu_j}x}\nu_j(y), \\
&&\\
\IEEEyesnumber
\label{gef1deriv}
\frac{\partial}{\partial x}\varphi_J'(\lambda,x,y) &=& \sum_{j\in J}   \left[ \frac{-i}{2\sqrt{\lambda-\mu_j} }-x\right]e^{-i \sqrt{\lambda-\mu_j} x}\nu_j(y) 
\\ && +\sum_{j, k\in J} \left[ i\sqrt{\lambda-\mu_k} S'_{J,j,k}(\lambda) + \frac{i  S_{J,j,k}(\lambda)}{2\sqrt{\lambda-\mu_k}}-xS_{J,j,k}(\lambda) \right]e^{i  \sqrt{\lambda-\mu_k} x}\nu_k(y)
\\ &&  +\sum_{j \notin  J}\left[ i\sqrt{\lambda-\mu_j}T'_j(\lambda) +\frac{iT_j(\lambda)}{2\sqrt{\lambda-\mu_j}}-xT_j(\lambda) \right]e^{i\sqrt{\lambda-\mu_j}x}\nu_j(y).
\end{IEEEeqnarray*}

At $0$ these two generalised functions become:

\begin{IEEEeqnarray*}{rLl} 
\IEEEyesnumber
\label{gefderivlam0}
\varphi_J'(\lambda,0,y)& =&\sum_{j\in J} \left( \sum_{k\in J} S'_{J,j,k}(\lambda) \nu_k(y) \right)+\sum_{j \in \mathbb{N} \backslash J} T'_j(\lambda) \nu_j(y), \\
&&\\
\frac{\partial}{\partial x}\varphi'(\lambda,0,y) &=&\sum_{j\in J} \left( \frac{1}{2i\sqrt{\lambda-\mu_j} }\nu_j(y)+\sum_{k\in J}\left[ i\sqrt{\lambda-\mu_k} S'_{J,j,k}(\lambda) - \frac{  S_{J,j,k}(\lambda)}{2i\sqrt{\lambda-\mu_k}} \right]\nu_k(y) \right)\\
\IEEEyesnumber
\label{gefderivlam0}
& &+\sum_{j \in \mathbb{N} \backslash J}\left( i\sqrt{\lambda-\mu_j}T'_j(\lambda) -\frac{T_j(\lambda)}{2i\sqrt{\lambda-\mu_j}} \right)\nu_j(y).
\end{IEEEeqnarray*}

We can now see that the external Neumann to Dirichlet map, denoted $\tilde{N}_1(\lambda),$ for the system will take the form
$$\tilde{N}_1(\lambda):l^2 \oplus l^2 \rightarrow l^2 \oplus l^2. $$
Using the $\tilde{N}(\lambda)$ introduced in Definition~\ref{intndm}, the fact that
$$\frac{\partial}{\partial x}\varphi_J'(\lambda,0,y) =\tilde{N}(\lambda)\varphi_J'(\lambda,0,y)+\frac{\partial}{\partial \lambda} \tilde{N}(\lambda)\varphi_J(\lambda,0,y)$$
in addition to the fact that
$$\frac{\partial}{\partial x}\varphi_J(\lambda,0,y) =\tilde{N}(\lambda)\varphi_J(\lambda,0,y),$$

we can see that by defining
$$ \tilde{N}_{(1)}(\lambda)=
 \begin{pmatrix}\tilde{N}(\lambda)&0\\
\frac{\partial}{\partial \lambda} \tilde{N}(\lambda)& \tilde{N}(\lambda)
   \end{pmatrix},$$

then, analogous to (\ref{kereq}), the Neumann to Dirichlet map on $E$ for this system will be $P \tilde{N}_{(1)}(\lambda)^{-1}.$
There is no reason for us to limit ourselves to first derivatives. We should go further now and do the same thing for $\varphi_J^{(n)}(\lambda,x,y).$ It is at this point the we take note of the fact that each successive differentiation of $e^{\pm i \sqrt{\lambda-\mu}x}$ respect to $\lambda$, produces a factor of of $x$. Since we will be focusing on $\varphi_J$ and $\frac{\partial}{\partial x} \varphi_J$ at the boundary, where $x=0,$ it is unnecessary to differentiate $e^{\pm i \sqrt{\lambda-\mu}x}$ more than once, and all terms that result from such actions, terms in these summands with a factor of $x^2$ will simply be denoted them as $\mathrm{h.o.t.}$
\\
\\

We will now introduce some new notation: $P\tilde{N}_{(n)}(\lambda)^{-1}$ and $PN_{(n)}(\lambda)$ we define to be the Neumann to Dirichlet maps for the system of $\varphi_J^{(n)}(\lambda,x,y)$ on the external and internal domains respectively.
Differentiating $\varphi_J^{(n)}(\lambda,x,y)$ $n$ times gives us:

\begin{IEEEeqnarray*}{rLl} 
\IEEEyesnumber
\label{geflam}
\varphi_J^{(n)}(\lambda,x,y) &=&\sum_{j\in J} \left[ -x \frac{\partial^n \tilde{N}_j}{\partial \lambda^n}(\lambda)+\mathrm{h.o.t}\right]e^{-i \sqrt{\lambda-\mu_j} x}\nu_j(y)\\
&&+\sum_{j,k\in J}\left[ x\sum_{q=1}^n \binom{n}{q} \frac{\partial^k \tilde{N}_k}{\partial \lambda^k}(\lambda).S_{J,j,k}^{(n-q)}(\lambda)+S_{J,j,k}^{(n)}(\lambda)+\mathrm{h.o.t}\right]e^{i \sqrt{\lambda-\mu_k} x}\nu_k(y) \\
&&+\sum_{j \notin J}\left[ x \sum_{q=1}^n \binom{n}{q} \frac{\partial^q \tilde{N}_j}{\partial \lambda^q}(\lambda).T_{j}^{(n-q)}(\lambda)+T_{j}^{(n)}(\lambda)+\mathrm{h.o.t} \right]e^{i\sqrt{\lambda-\mu_j}x}\nu_j(y). 
\end{IEEEeqnarray*}

These expressions evaluated on the boundary become:

\begin{IEEEeqnarray*}{rLl} 
\IEEEyesnumber
\varphi_J^{(n)}(\lambda,0,y) &=&\sum_{j,k\in J} S_{J,j,k}^{(n)}(\lambda)\nu_k(y) +\sum_{j \notin J} T_{j}^{(n)}(\lambda)\nu_j(y), \\
&&\\
\IEEEyesnumber
\label{gefderiv}
\frac{\partial}{\partial x}\varphi_J^{(n)}(\lambda,0,y) &=& \sum_{j\in J} \left[ -\frac{\partial^n \tilde{N}_j}{\partial \lambda^n} (\lambda) \right]\nu_j(y) \\
&& \sum_{j,k\in J}\left[ \sum_{q=0}^n \binom{n}{q} \frac{\partial^q \tilde{N}_k}{\partial \lambda^q}(\lambda).S_{J,j,k}^{(n-q)}(\lambda)\right]\nu_k(y) \\
&&+\sum_{j \notin J}\left[ \sum_{q=0}^n \binom{n}{q} \frac{\partial^q \tilde{N}_j}{\partial \lambda^q}(\lambda).T_{j}^{(n-q)}(\lambda)\right]\nu_j(y). 
\end{IEEEeqnarray*}

$\tilde{N}_{(n)}(\lambda)$ will be a block-upper triangular matrix, ($\tilde{N}_j,$ for $ j \in \mathbb{N}$ has been used above to be the $j^{th}$ element of the diagonal matrix 
$\tilde{N}$ so the reader should avoid confusing the two).
$$\tilde{N}_{(n)}(\lambda):\bigoplus_{\lbrace 0,\ldots ,n\rbrace} l^2 \longrightarrow \bigoplus_{\lbrace 0,\ldots ,n\rbrace} l^2,$$
given by
$$\tilde{N}_{(n)}(\lambda)= \begin{pmatrix}  \tilde{N}(\lambda) &0& \cdots &\cdots&  \cdots&0 \\
  \frac{\partial}{\partial \lambda} \tilde{N}(\lambda)& \tilde{N}(\lambda) &0& \cdots &\cdots &0\\
  \binom{2}{0} \frac{\partial^2}{\partial \lambda^2}\tilde{N}(\lambda)  &\binom{2}{1}\frac{\partial}{\partial \lambda}\tilde{N}(\lambda)  &\binom{2}{2}\tilde{N}(\lambda)   &0&\cdots& 0\\
  \cdots & \cdots &\cdots &\cdots&\cdots &0  \\
    \cdots & \cdots &\cdots &\cdots&\cdots &0  \\
\binom{n}{0}\frac{\partial^n}{\partial \lambda^n}\tilde{N}(\lambda)  &\binom{n}{1} \frac{\partial^{n-1}}{\partial \lambda^{n-1}} \tilde{N}(\lambda)  &\cdots&\cdots  &\cdots& \binom{n}{n}\tilde{N}(\lambda)
   \end{pmatrix}$$

\subsection{A Neumann to Dirichlet map for the system on the internal domain}
Now we turn our attention towards the internal domain $X.$ $N_{(n)}(\lambda)$ can be computed with an extension of Levitin-Matletta's method. Equation (\ref{alternate}) describes the method when used to compute $N(\lambda)$. Their ``trick" to increase the rate of convergence (\ref{levstrick}) is unaffected, so long as it is performed only once.
\\
\\
The $\lbrace U_m \rbrace$ remain unchanged from (\ref{alternate}) and are used to denote the orthonormal Neumann eigenfunctions and $\mu_m,$ their corresponding eigenvalues. $\lbrace \nu_j \rbrace$ will denote orthonormal an basis of $L^2(\Gamma)$  and finally, we will use $\lbrace \Phi_k \rbrace$ and $\lbrace \Phi^{(n)}_k \rbrace,$ to denote solutions to the following system:

\begin{IEEEeqnarray*}{llll} 
\IEEEyesnumber \label{xfirstsystem}
(\Delta - \lambda)\Phi_{k}(\lambda,x,y)=0, &\quad \quad   &  \frac{\partial \Phi_{k}}{\partial n} |_{\Sigma}=0 , \ \frac{\partial \Phi_{k}}{\partial n}|_{\Gamma} = \nu_{k_0}\\
(\Delta-\lambda)\Phi_{k}'(\lambda,x,y)-\Phi_{k}(\lambda,x,y)=0, &  & \frac{\partial \Phi_{k}'}{\partial n} |_{\Sigma}=0 , \ \frac{\partial \Phi_{k}'}{\partial n}|_{\Gamma} = \nu_{k_1}\\
\quad \vdots & & \quad \vdots \\
(\Delta-\lambda)\Phi_{k}^{(n)}(\lambda,x,y)-\Phi_{k}^{(n-1)}(\lambda,x,y)=0, & &  \frac{\partial \Phi_{k}^{(n)}}{\partial n} |_{\Sigma}=0 , \ \frac{\partial \Phi^{(n)}_{k}}{\partial n}|_{\Gamma} = \nu_{k_n},
\end{IEEEeqnarray*}
where k is the n-tuple $\lbrace k_0,k_1,\ldots k_n \rbrace.$
\\
\\
% We already know that $N_{,k,l}= \ip{\Phi^{(n)}_k|_{\Gamma},\nu_l}= \ip{\Phi^{(n)}_k|_{\Gamma},\frac{\partial \Phi^{(n)}_l}{\partial_n}_\Gamma}.$ Integration by parts gives. $b\leq a$
We begin by fixing $\lambda$ and will now calculate elements of $N_{(n)}(\lambda)$ which map to $\Phi^{(n)},$ where $k_n\in \mathbb{N}$ and $l\in \mathbb{N}^n:$

\begin{IEEEeqnarray*}{lCll} 
N_{(n),k_n,l}& =&\ip{N \nu_{k_n},\nu_{l}}= \ip{\Phi^{(n)}_k|_{\Gamma},\nu_{l}}=\ip{\Phi^{(n)}_{k}|_{\Gamma},\frac{\partial \Phi^{(n)}_{l}}{\partial_n}_\Gamma}\\
&=&\ip{\nabla\Phi^{(n)}_{k},\nabla \Phi^{(n)}_{l}}+\ip{\Delta \Phi^{(n)}_{k},\Phi^{(n)}_{l}}=\ip{\nabla\Phi^{(n)}_{k},\nabla \Phi^{(n)}_{l}}+\lambda \ip{\Phi^{(n)}_{k},\Phi^{(n)}_{l}}+\ip{\Phi^{(n-1)}_{k},\Phi^{(n)}_{l}}.
\end{IEEEeqnarray*}

Since each $\Phi_k$ can be written as $\sum_m \Phi_k\ip{\Phi_k,U_m}$ (the same is true for  $\Phi^{(n)}_k$), and by definition, $\Phi_k =(\Delta-\lambda)^{-j}\Phi^{(n-j)},$ it follows that:

\begin{IEEEeqnarray*}{lLl} 
N_{(n),k_n,l} &=\sum_m(\ip{\nabla U_m,\nabla U_m}+\lambda)\ip{\Phi^{(n)}_{k},U_m}\ip{U_m,\Phi^{(n)}_{l}}+\ip{\Phi^{(n-1)}_{l},\Phi^{(n)}_{l}}\\
&=\sum_m(\ip{\nabla U_m,\nabla U_m}+\lambda)\ip{\Phi^{(n)}_{k},U_m}\ip{U_m,\Phi^{(n)}_{l}}+\sum_m\ip{\Phi^{(n-1)}_{k},U_m}\ip{U_m,\Phi^{(n)}_{l}}\\
\IEEEyesnumber \label{derivndmaplm}
 &=\sum_m(\lambda-\mu_m)\ip{\Phi^{(n)}_{k},U_m}\ip{U_m,\Phi^{(n)}_{l}}+\sum_m \ip{\Phi^{(n-1)}_{k},U_m}\ip{U_m,\Phi^{(n)}_{l}}.
 \end{IEEEeqnarray*}
Green's second identity, for any $j,$ means that
\begin{IEEEeqnarray*}{lLl} 
\IEEEyesnumber \label{greenndm}
\ip{\Phi_j,U_m} &=\sum_m\frac{1}{\lambda-\mu_m}\left( \ip{\Delta \Phi_j,U_m}-\ip{\Phi_j,\Delta U_m} \right)=\sum_m\frac{1}{\lambda-\mu_m}\ip{\nu_j,U_m|_\Gamma}\\
\ip{\Phi^{(n)}_{j},U_m}&=\sum_m\frac{1}{\lambda-\mu_m}\left( \ip{\Delta \Phi^{(n)}_{j},U_m}-\ip{\Phi^{(n)}_{j},\Delta U_m}-\ip{\Phi^{(n-1)}_{j},U_m} \right).\\
&=\sum_m\frac{1}{\lambda-\mu_m}\left( \ip{\nu_{j_n},U_m|_\Gamma}-\ip{\Phi^{(n-1)}_j,U_m} \right).
\end{IEEEeqnarray*}
So then $N_{(n),k,l}$ becomes
$$\sum_m\left( \ip{\nu_{k},U_m|_\Gamma}-\ip{\Phi^{(n-1)}_j,U_m} \right)\ip{U_m,\Phi^{(n)}_{l}},$$
and inductively, we see that
$$N_{(n),k_n,l}=\sum_m \ip{\nu_{k_n},U_m|_\Gamma}\ip{U_m,\Phi^{(n)}_{l}}=\sum_m \sum_{p=1}^n \frac{(-1)^{p-1}}{(\lambda-\mu_m)^p}  \ip{\nu_{k_n},U_m|_\Gamma} \ip{U_m|_\Gamma,\nu_{l_p,}}.$$
Thus the $n^{th}$ block-row for $N_{(n)}(\lambda)$ is made up of the direct sum of maps defined component-wise by

\begin{equation}
\label{etaequation}
\eta_{(p)}(\lambda)= \sum_m  \frac{(-1)^{p-1}}{(\lambda-\mu_m)^p}  \ip{\nu_{k},U_m|_\Gamma} \ip{U_m|_\Gamma,\nu_{l,}},
\end{equation}
where $p$ runs through $1,\ldots n.$
\\
\\
Now, for the system, $N_n(\lambda)$ is a block upper-triangular matrix, acting on $\Phi \oplus \cdots \oplus \Phi^{(n)},$ of the same form as $\tilde{N}_{(n)}(\lambda),$ given by

\begin{equation}
\label{etamatrix}
N_{(n)}(\lambda)= \begin{pmatrix} \eta_{(1)}(\lambda) &0& \cdots &\cdots&  \cdots&0 \\
\eta_{(2)}(\lambda)&\eta_{(1)}(\lambda) &0& \cdots &\cdots &0\\
\eta_{(3)}(\lambda)  &\eta_{(2)}(\lambda) &\eta_{(1)}(\lambda) &0&\cdots& 0\\
  \cdots & \cdots &\cdots &\cdots&\cdots &0  \\
    \cdots & \cdots &\cdots &\cdots&\cdots &0  \\
\eta_{(n)}(\lambda)  &\eta_{(n-1)}(\lambda) &\cdots&\cdots  &\cdots&\eta_{(1)}(\lambda)
   \end{pmatrix}
\end{equation}

We can assert here that the computational cost of such an approach is no greater than that of crudely approximating the derivatives from first principles. Each of the $\eta_p$ in (\ref{etaequation}), that goes into making up the $N_{(n)}$ matrix (\ref{etamatrix}), comes about in a very similar way to (\ref{levstrick}) and has the same computational cost, namely that of $n \times m$ by $m \times n$ matrix multiplication. Even though the number of $\eta_p$ required to make up the block-diagonal $N_{(n)}$ rises arithmetically with the order of the derivatives, the number of unique $\eta_p$ required for the $n^{th}$ derivative is only $n+1$. The reader will see in the following subsection, that the rest of the calculation consists of a series of additions, multiplications, inversions and one Singular Vaule Decomposition of (in practice) much smaller matrices. There are no issues relating to the speed of convergence either as the order of the denominators in (\ref{etaequation}) will always be greater or equal to $2$.

\subsection{Extracting $S_J^{(n)}(\lambda)$}
The coefficients of $\varphi^{(n)}_J(\lambda,0,y)$ can now be computed by finding a basis for the null space of 
$$L_n=\left( PN_{(n)}- P\tilde{N}_{(n)}^{-1} \right).$$
We will denote such a basis as ${\omega_1^{(n)},\ldots \omega_J^{(n)}}.$
Finally, using the same argument made when calculating $S(\lambda),$
$$\tau_n(\lambda)=\lbrace(1-P)\omega_1^{(n)},\cdots ,(1-P)\omega_J^{(n)}\rbrace \lbrace(1-P)N \omega_1^{(n)},\cdots ,(1-P)N\omega_J^{(n)}\rbrace^{-1},$$
thus 
\begin{IEEEeqnarray*}{lCll} 
\tau_n(\lambda)& =&\left( D(\lambda,n)-\sum_{q=0}^n \binom{n}{q} D(\lambda, q).S^{(n-q)}(\lambda) \right)S^{(n)}(\lambda)^{-1}\\
&=&\left( D(\lambda, n) -\sum_{q=1}^n\binom{n}{q} D(\lambda, q).S^{(n-q)}(\lambda)- D(\lambda, 0)S^{(n)}(\lambda) \right)S^{(n)}(\lambda)^{-1}.\
\end{IEEEeqnarray*}
All of the $S_J^{(n-q)}(\lambda)$ are known, having been previously calculated. Finally we can say that
\begin{equation}
\label{finalderiv}
S_J^{(n)}(\lambda)=\Bigg(\mathrm{Id}-D(\lambda, 0)\tau_n(\lambda)\Bigg)^{-1}. \left( D(\lambda, n) .\tau_n(\lambda)-\sum_{q=1}^n \binom{n}{q} D(\lambda, q). \tau_n(\lambda).S^{(n-q)}(\lambda) \right).
\end{equation}

\section{Embedded eigenvalues and resonances}
\label{resbackground}

This section will give an overview of embedded eigenvalues, complex resonances and describe how the scattering matrix can be used to compute the latter. It should be noted that computational methods are becoming more important generally in science, with new techniques under continuous development~\cite{steve_johnson_casimir}, \cite{timobem}, \cite{ab_mps}.
\subsection{Embedded eigenvalues}

An embedded eigenvalue of an operator is an isolated point within an operator's continuous spectrum which is actually an eigenvalue of finite multiplicity. The paper by Evans, Levitin and Vassiliev~\cite{trappedmodes} proved the existence of embedded eigenvalues for the Neumann Laplacian on two dimensional waveguides with an obstacle, and/or deformation of the waveguide so long as the domain has cross-sectional symmetry. 

This was further generalised  to waveguides with cylindrical ends by Davies and Parnovski~\cite{trappedmodes2}. Parnovski and Levitin have, amongst others, produced two other papers on this topic~\cite{hawkparnovski}~\cite{pencilev}. More recently, in Levitin and Strohmaier's paper~\cite[6.3]{alexnew} they describe a technique to calculate embedded eigenvalues numerically by looking for real values of $\lambda$ that make the sub-matrix of $\left( N(\lambda)- \tilde{N}^{-1}(\lambda) \right)$, obtained by omitting the rows and columns representing non square integrable modes, singular  (see (\ref{levstrick}) and definition ~\ref{intndm} for definitions of $N$ and $ \tilde{N}$). In this case, non square integrable modes will be the $\mu< \lambda,$ as in such a situation  $e^{i \sqrt{\lambda-\mu}}$ will certainly not be in $L^2$ for real $\lambda$. An eigenvector corresponding to a zero-eigenvalue will necessarily represent the boundary data of an eigenfunction and not a generalised eigenfunction as the Neumann to Dirichlet maps on the internal and external domains co-inside when acting on that vector and the modes have been chosen deliberately to be in $L^2$. The problem with such an approach is that it is hard to distinguish an embedded eigenvalue from a resonance very close to the real line~\cite[6.3]{alexnew}. Like Levitin and Strohmaier, we have calculated a small number of embedded eigenvalues, but the main focus of the numerical calculations in section~\ref{numresults} has been to use the scattering matrix to locate resonances. The reason for this is that the properties of the scattering matrix lend themselvs to use in algorithms that are able to ``sweep" an area in search of them and this process can be automated to ``zoom in" on any resonances found (see Proposition~\ref{polesint} and the preceeding and following discussion).

\subsection{Complex resonances}
The scattering matrix, and its derivatives, can be used to calculate complex resonances on any sheet of $Z$. Previous work on this topic has focused on on either providing asymptotic bounds for the number of resonances~(\cite{christiensen}, \cite{melrosepolybounds}, \cite{zworskipolybounds}, \cite{tanyaasymp}, \cite{zworskig}), or their calculation~(~\cite{lm}, \cite{aslanyanebdcalc}, \cite{apv}, \cite{db}, \cite{betcke1}, \cite{betcke2}, \cite{nazarov1}, \cite{alexnew}). For the remainder of this paper, we will focus on the latter. As was alluded to previously, reason we have gone down this route is because the numerical technique we have just constructed is absolutely conducive to this. Levitin and Strohmaier have already used the scattering matrix to locate resonances on surfaces with hyperbolic cusps~\cite{alexnew}; as we now are in possession of the scattering matrix for Euclidean waveguides, we will do the same for them.
\\
\\
We define a resonance here to be a pole of the scattering matrix. The relationship between the resolvent, the Neumann to Dirichlet map and the scattering matrix means that poles of the resolvent coincide with zeros of the determinant of the the inverse of the scattering matrix, and their multiplicities will be the same.
\\
\\
Theorem \ref{funceq} can be used to show that every pole of $S_J(\lambda),$ on the sheet $J$ of $Z,$ coincides with a zero of $S_J(\lambda^*)$ in the physical sheet and vice versa. We have used $\lambda^*$ to denote the canonical projection of $\lambda$ to the physical sheet; when both $\lambda$ and $\lambda^*$ are identified with a subset of the complex plane, they will be in the same location. Since the resolvent and scattering matrix are holomorphic in the physical sheet, it can't have poles there, and we will have no zeros in a non-physical sheet of $Z.$ This means we can now make use of the argument principle to locate resonances, and locating resonances in a non-physical sheet of $Z$ has been reduced to locating zeros in the physical sheet. The argument principle together with the Jacobi formula gives:

\begin{prop}
\label{polesint}
Let $C$ be a contour in a non-physical sheet of $Z, \ J$ with winding number one and let $\#$ be the counting function $\#: \ C\mapsto \mathbb{N}$ which counts the number of poles enclosed by $C.$ Then
$$\#(C)=\frac{1}{2 \pi i}\oint_C\mathrm{Tr}(S_J(\lambda)^{-1} .S_J'(\lambda)).$$
\end{prop}
This approach was featured in the paper by Davies and Aslanyan, but not applied to the scattering matrix~\cite{aslanyanebdcalc}. If a contour can be found that contains one or more zeros, we can subdivide then integrate over the subdivisions and repeat the process until a small enough contour has been found containing a single resonance. Simpson's method for numerical integration will suffice for this. Once we are sufficiently close to any resonances we may find, Newton's method can then be used to obtain their location to a desired level of accuracy. We can then multiply the scattering matrix $S_J(\lambda) \text{ by } (\lambda_0-\lambda)^{-1},$ where $\lambda_0$ is the location of the zero, then apply Newton's method again, repeating if necessary to find  its order. This process can be automated to such an extent that one need only feed parameters for a search area into a script, wait, and then be rewarded with a list of locations for any resonances found.

\section{Conclusion and presentation of results}
\label{numresults}

We will conclude by presenting the reader with the results of some numerical calculations. In \ref{circobs} we will replicate and improve somewhat the accuracy of the calculations produced by Aslanyan, Parnovski and Vassiliev~\cite{apv}, and Levitin and Marletta~\cite{lm}. In~\ref{mraccuracy} - \ref{modeaccuracy} we will discuss the finite element method mesh refinement, number of eigenvalues and and modes needed to obtain a desired accuracy and discuss where to ``cut" the internal domain from the ends. All resonances up to this point will have been calculated for the sheet $J=\lbrace 1 \rbrace$ (see~\ref{jay}), but in \ref{newsheet}, we will track the location of a resonance and observe it cross from the sheet $J=\lbrace 1 \rbrace$ to $J=\lbrace 2 \rbrace$ as the domain is changed. In \ref{circends} we will present some tabulated results and graphs describing the motion of resonances in an area as the domain is continuously changed. \ref{tdsl} gives a brief overview of time delay and scattering length and we will finalise the paper by linking to some animations of plots of various resonances as the domain is continuously changed before suggesting additional areas of study.
\\
\\
All of the finite element method data used to obtain these results has been produced in FreeFem++~\cite{freefem}. FreeFem++ was chosen at the time because it is free and open source, with no restrictions on installation and use. It has proved itself to be reliable and with good documentation.  FreeFem++ contains methods to produce triangulations (or tetrahedralisation if working in 3-D) for domains given to it by the user in the form of parameterised curves, implements both the finite element method for solving PDEs and the implicitly restarted Arnoldi method for finding eigenvalues and eigenvectors by making use of the ARPACK subroutine~\cite{arpack} ~\cite{arnoldi}. Though the results have not been presented here, it was able to perform similar calculations in three dimensions for ends with rectangular cross sections and circular cross-sections and could theoretically work for domains with arbirtarily drawn cross sections and domains with potentials. If the reader wishes to compare finite element method software, the site \url{http://feacompare.com} has an exhaustive list that, at the time of writing, numbered 78. Boundary element method techniques could also be used in place of the finite element method here. In certain circumstances they may have several advantages~\cite{timobem}. At this stage, we have not explored this option any further because, for each domain, calculation of numerical solutions to pdes is only required once, and for the two dimensional domains we have presented results for here, the compute time has been ``tolerable", around five minutes for each domain. This will certainly change if higher dimensional domains are used with more complicated geometries. The method of particular solutions~\cite{betke_mps} is another computational technique that has been generating significant interest of late, and where software to implement it is starting to become easily available (see \url{https://github.com/ahbarnett/mpspack}).
\\
\\
The code used to extract the scattering matrix and its derivatives, from the output of the finite element method (then later search for resonances) was written in Wolfram Mathematica and has been subsequently rewritten in Sage. It is available at \\ \url{https://bitbucket.org/greg5783/smcalc2/}. It was important to use a software package that was able to perform linear algebaic operations on vectors and matrices with complex values. MATLAB could also be used for this. Sage was found to have a slight speed advantage over Mathematica provided care was taken to write efficient code. Sage is more complicated to use, for example when parallelising operations (in mathematica, one only has to decorate the entire script with \verb!parallel()! and parralelisation of any and all ensuing operations is implemented automatically in an optimised way.

\subsection{Cylinders with a circular obstacle}
\label{circobs}

\begin{figure}[H]
\centering
\includegraphics[trim = 10mm 0mm 20mm 140mm, clip, width=6cm]{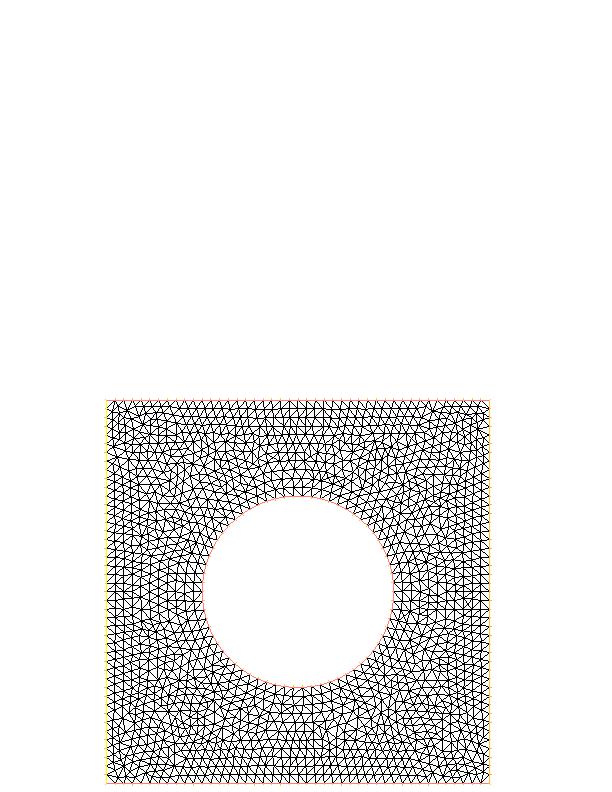}
\caption{One of our triangulated interior domains produced by FreeFEM++~\cite{freefem}. The waveguide is composed of this interior domain, with the two ends the same width as the interior domain joined on the left and right sides.}
\label{levitinsrectangle}
\end{figure}

The first thing we wanted to do once we had a working algorithm was replicate the results of our predecessors and try to improve their accuracy. Amongst other things, Levitin and Marletta looked at an infinite rectangle of width $2$, with a single circular obstruction (See Figure~\ref{levitinsrectangle}). The radius of this obstruction, $R$, is varied along with  $\delta,$ the position of its centre (vertical displacement) relative to the centre line of the rectangle. When the vertical displacement is $0$, there exist embedded eigenvalues. The embedded eigenvalues decay to a resonances when the vertical displacement $\delta$ becomes non-zero \cite{lm}. With the parameterisation of $\lambda\mapsto \lambda^2,$ they presented a number of values for these resonances. We have performed our calculation to the highest accuracy our method and available hardware afforded us, namely $2000$ eigenvalues, $20$ modes (see~\ref{evaccuracy}) and a mesh refinement of $80 $ (see~\ref{mraccuracy}), and have displayed ours and our predecessors' results in Table~\ref{accuracytable} below. Any additional eigenvalues, modes and mesh refinements after these values did not affect the results. We have been able to offer a slight improvement on the number of decimal places.
 
\begin{figure}[H]
\begin{tabularx}{\textwidth}{| c | c || X  | X | X   |    }
\hline
 $R$ & $ \delta$ & Our calculation & Levitin-Marletta & Aslanyan et al\\
 \hhline{|=|=|=|=|=|}
 
$0.3$&$ \delta=0$   & $1.504\light{97}$ &$1.50486$ & $ 1.5048$\\ \hline
$0.3$&$\delta=0.1$  & $1.50783 + 0.0001205i$& $1.5078 +   10^{-4} i$& $1.5102+ \times 10^{-4}i$\\  \hline
$0.3$&$\delta=0.2$  & $1.51651 + 0.0004740i$&$1.5165 + 5 \times 10^{-4} i$ & $1.5188 = 5 \times 10^{-4}i$ \\ 
 \hhline{|=|=|=|=|=|}

 $0.5$&$ \delta=0$   & $1.391\light{38}$ &$1.39134$ & $ 1.3913$\\ \hline
$0.5$&$\delta=0.1$  & $1.39785 + 0.0009255i$& $1.3979 + 9 \times 10^{-4} i$& $1.3998+9 \times 10^{-4}i$\\  \hline
$0.5$&$\delta=0.2$  & $1.41779 + 0.0039101i$&$1.4178 + 3.90\times 10^{-3} i$ & $1.4196+3.93 \times 10^{-3}i$ \\  \hline
   \end{tabularx}
\captionof{table}{This table displays our own results alongside those of Levit-Marletta and Aslanyan, Parnovski and Vassiliev. Greyed out figures portray a lower level of confidence in their accuracy on our part.}
\label{refine}
\end{figure}

\subsection{Some notes on mesh refinement}
\label{mraccuracy}  
  
   \begin{figure}[H]
   \label{refinement}
\centering
 \begin{tabular}{cc   }
\includegraphics[trim = 10mm 0mm 20mm 140mm, clip, width=6cm]{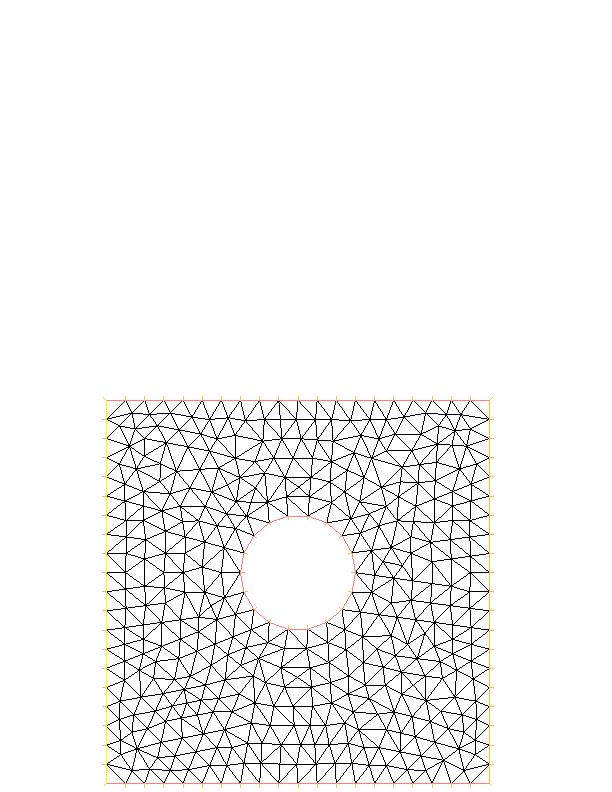}&\includegraphics[trim = 10mm 0mm 20mm 140mm, clip, width=6cm]{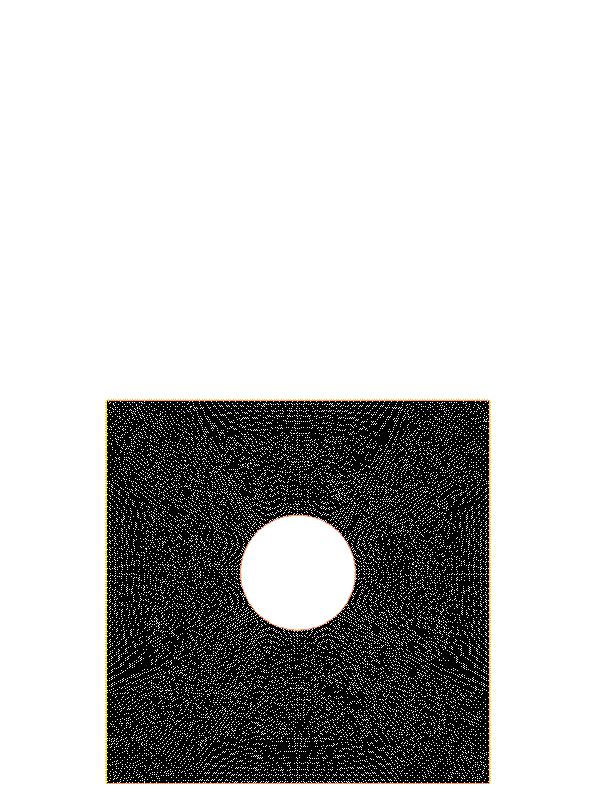}
\end{tabular}
\caption{Some internal domains for $R=0.3, \ \delta=0.1$ with mesh refinement of $10$ and $80$ respectively.}
\label{accuracytable}
\end{figure}

We have looked at the effects of mesh refinements on the accuracy of the results obtained. Whilst it should be obvious that the more refined the mesh becomes, the more accurate the result, we have tabulated the results of some experiments to demonstrate just to what extent. The standard method for producing a triangulated domain in FreeFem++ is to draw the outline as a union of parameterised curves, then use the program's own triangulation algorithm after specifying the number of points on each such curve. Our scale of mesh refinement was taken to be the number of such points per unit length on the boundary. Using the same domain as~\ref{levitinsrectangle}, we have tabulated some results for the calculation of these resonances using $2000$ eigenvalues and $20$ modes for each connected component of $\Gamma$ (see Figure~\ref{accuracytable}). The reader can see from Table~\ref{mr1} and Table~\ref{mr2} how the result stabilises as the mesh refinement increases.

\begin{figure}[H]
\begin{tabularx}{\textwidth}{| c  ||X | X  |  }
\hline
Mesh refinement&$R=0.3 \ \delta=0.1$  & $R=0.3 \ \delta=0.2$\\
 \hhline{|=||=|=|}
 $10$&$1.50943+0 .0001157 i$&$1.51791 + 0.0004530i$\\
  $15$&$1.50847 +0 .0001185 i$&$1.51708 + 0.0004657i$\\
   $20$&$1.50821 +   0.0001193i $&$1.51684 +  0 .0004691i$\\
    $25$&$1.50805 +0 .0001198 i$&$1.51670 +0 .0004712i$\\ 
     $30$ &$1.50797 +   0 .0001200i$&$1.51663 +    0.0004722i$\\
    $35$&$1.50793 +0 .0001202i $&$1.51660 +0 .0004727i$\\
     $40$&$1.50790 + 0 .0001203 i$&$1.51657 +   0 .0004731i$\\
   $45$&$1.50788 +0 .0001203 i$&$1.51655 + 0.0004734i$\\
   $50$&$1.50786 +  0  .0001204i $&$1.51654 +   0 .0004736i$\\
    $55$&$1.50785 + 0.0001204 i$&$1.51653 + 0.0004737i$\\
    $60$&$1.50785 +0 .0001204 i$&$1.51652 + 0.0004738i$\\ 
      $65$&$1.50784 + 0.0001204i$&$1.51651 + 0.0004739i$\\
   $70$&$1.50783 + 0.0001205i$&$1.51651 + 0.0004740i$\\
   $75$&$1.50783 +   0.0001205i$&$1.51651 +   0.0004740i$\\
   $80$&$1.50783 + 0.0001205i$&$1.51651 + 0.0004740i$\\ \hline
   \end{tabularx}
\captionof{table}{This table shows the location of the resonances calculated as the number of mesh refinements is increased. The domain is unchanged from~\ref{levitinsrectangle}, $R$ denotes the radius of the obstacle and $\delta,$ the distance it is off-centre by.}
\label{mr1}
\end{figure}

\begin{figure}[H]
\begin{tabularx}{\textwidth}{| c  ||X | X  |  }
\hline
Mesh refinement&$R=0.5 \ \delta=0.1$&$R=0.5 \ \delta=0.2$\\
 \hhline{|=||=|=|}
 $10$&$1.39874 +0 .0009122i$&$1.41857 + .00385170i$\\
  $15$&$1.39822 +0 .0009199i$&$1.41811 +0.0038854i$\\
   $20$&$1.39805 +0  .0009225i$&$1.41796 +   0.0038971i$\\
    $25$&$1.39797 +0 .0009237i$&$1.41789 +0 .0039021i$\\ 
     $30$ &$1.39793 +   0.0009243i$&$1.41786 +   0.0039049i$\\
    $35$&$1.39791 + 0.0009247i$&$1.41784 + 0.0039063i$\\
     $40$&$1.39789 +  0 .0009249i$&$1.41782 +   0.0039075i$\\
   $45$&$1.39788 + 0.0009251i$&$1.41781 + 0.0039083i$\\
   $50$&$1.39787 +  0 .0009252i$&$1.41781 +   0.0039089i$\\
    $55$&$1.39786 +0.0009253 i$&$1.41780 +0 .0039092i$\\
    $60$&$1.39786 + 0.0009254i$&$1.41780 +0 .0039095i$\\ 
     $65$&$1.39786 + 0.0009254i$&$1.41779 + 0.0039097i$\\
   $70$&$1.39785 + 0.0009255i$&$1.41779 + 0.0039098i$\\
   $75$&$1.39785 +  0.0009255i$&$1.41779 +   0.0039101i$\\
   $80$&$1.39785 + 0.0009255i$&$1.41779 + 0.0039101i$\\ \hline
   \end{tabularx}
\captionof{table}{Like Table~\ref{mr1}, this table shows the location of the resonances calculated as the number of mesh refinements is increased. The domain is still unchanged from~\ref{levitinsrectangle}, $R$ still denotes the radius of the obstacle and $\delta,$ the distance it is off-centre by.}
\label{mr2}
\end{figure}

\subsection{Some notes on the number of eigenvalues}
\label{evaccuracy}
In subsection~\ref{circobs}, we mentioned that we used $2000$ eigenvalues but gave no justification for this. Here we will present the reader with some graphs and charts to demonstrate convincingly, why we chose this number. We have tabulated data describing what happens when number of eigenvalues is increased. We also present some convergence graphs that track the value of the leading coefficient of the scattering matrix as the number of eigenvalues increases for a selection of the domains at a randomly chosen point.

\begin{figure}[H]
 \begin{tabularx}{\textwidth}{| c  ||X | X  |  }
\hline
Number of eigenvalues&$R=0.3 \ \delta=0.1$  & $R=0.3 \ \delta=0.2$\\
 \hhline{|=||=|=|}
 $200$&$1.50783 +0.000120482i$&$1.51651 +0.000474043i$\\
 $300$&$1.50783 +0.000120482i$&$1.51651 +0.000474043i$\\
 $400$&$1.50783 +0.000120482i$&$1.51651 +0.000474045i$\\
 $500$&$1.50783 +0.000120482i$&$1.51651 +0.000474045i$\\
 $600$&$1.50783 +0.000120482i$&$1.51651 +0.000474045i$\\
 $700$&$1.50783 +0.000120482i$&$1.51651 +0.000474045i$\\
 $800$&$1.50783 +0.000120483i$&$1.51651 +0.000474045i$\\
 $900$&$1.50783 +0.000120482i$&$1.51651 +0.000474045i$\\
 $1000$&$1.50783 +0.000120483i$&$1.51651 +0.000474046i$\\
 $1200$&$1.50783 +0.000120483i$&$1.51651 +0.000474046i$\\
 $1400$&$1.50783 +0.000120483i$&$1.51651 +0.000474046i$\\
 $1600$&$1.50783 +0.000120483i$&$1.51651 +0.000474046i$\\
 $1800$&$1.50783 +0.000120483i$&$1.51651 +0.000474046i$\\
 $2000$&$1.50783 +0.000120483i$&$1.51651 +0.000474046i$\\
  \hline
   \end{tabularx}
\captionof{table}{This table shows the effect of increasing the number of eigenvalues on the accuracy of some selected resonances. We have used a mesh refinement of $80$ here and $20$ modes. The domain is from~\ref{levitinsrectangle}, $R$ denotes the radius of the obstacle and $\delta,$ the distance it is off-centre by.}
\end{figure} 
  
\begin{figure}[H]
\begin{tabularx}{\textwidth}{| c  ||X | X  |  }
\hline
Number of eigenvalues&$R=0.5 \ \delta=0.1$&$R=0.5 \ \delta=0.2$\\
 \hhline{|=||=|=|}
 $200$&$1.39785 +0.000925529i$&$1.41779 +0.00391010i$\\
 $300$&$1.39785 +0.000925538i$&$1.41779 +0.00391011i$\\
 $400$&$1.39785 +0.000925536i$&$1.41779 +0.00391013i$\\
 $500$&$1.39785 +0.000925536i$&$1.41779 +0.00391013i$\\
 $600$&$1.39785 +0.000925536i$&$1.41779 +0.00391014i$\\
 $700$&$1.39785 +0.000925536i$&$1.41779 +0.00391014i$\\
 $800$&$1.39785 +0.000925537i$&$1.41779 +0.00391014i$\\
 $900$&$1.39785 +0.000925536i$&$1.41779 +0.00391014i$\\
 $1000$&$1.39785 +0.000925537 i$&$1.41779 +0.00391014i$\\
 $1200$&$1.39785 +0.000925537i$&$1.41779 +0.00391014i$\\
 $1400$&$1.39785 +0.000925537i$&$1.41779 +0.00391014i$\\
 $1600$&$1.39785 +0.000925537i$&$1.41779 +0.00391014i$\\
 $1800$&$1.39785 +0.000925537i$&$1.41779 +0.00391014i$\\
 $2000$&$1.39785 +0.000925537i$&$1.41779 +0.00391014i$\\
\hline
   \end{tabularx}
\captionof{table}{This table shows the effect of increasing the number of eigenvalues on the calculated position of some selected resonances. The domain remains the same, $R$ denotes the radius of the obstacle and $\delta,$ the distance it is off-centre by.}
\end{figure}

From this it might seem like it is unnecessary to use many eigenvalues, however the number of eigenvalues does have a significant impact on the coefficients of the scattering matrix as can be seen in Figure~\ref{convgraphs}. We have picked, as an example, the domain where $R=0.3$ and $\delta=0.1$ with a mesh refinement of $80.$  We have plotted the real and imaginary components of the leading coefficient of the scattering matrix at the value $1+0.1i.$ This is typical behavior for any arbitrarily chosen point. 
   \begin{figure}[H]
\centering
 \begin{tabular}{cc   }
\includegraphics[width=7cm]{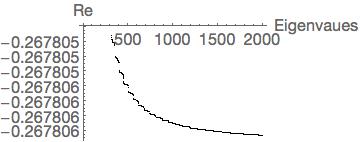}&\includegraphics[width=7cm]{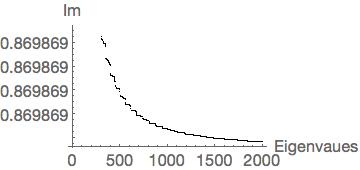}
\end{tabular}
\caption{The real and imaginary components of the first coefficient of the scattering matrix plotted against the number of eigenvalues and eigenvectors used to compute it.}
\label{convgraphs}
\end{figure}

\subsection{Some notes on the ``length" of the domains}
\label{ldomaccuracy}
The reader should recall that the waveguide is made up of a compact part, on which we have applied the finite element method, and the non-compact ends. The choice as to where the compact part finishes and the ends begin is somewhat arbitrary. It does, however, have an impact on the accuracy of any calculations. In practice, if the compact part if taken to be too large, the internal Neumann to Dirichlet map becomes ill conditioned and a source of major inaccuracy. If the compact part is cut too short then the effect of the obstacle is not completely captured. We settled on the lengths we did after careful experimentation. It should be noted here that all domains for which we have furnished results are comparatively simple, so performing these kind of experiments are ``affordable" in terms of compute time; for more complicated and higher dimensional domains, it is a potential vulnerability of the method and extreme care must be taken to ensure that the matrix inversions are well conditioned.
   
\subsection{The importance of the number of modes and choice of auxiliary point}
   \label{modeaccuracy}

It should be noted that the number of modes and the auxiliary point (the $\tilde{\lambda}$  from~\ref{levstrick} used to speed up convergence) chosen had an undetectable effect on the accuracy of the calculations, so long as the choice was ``sensible". It should also be noted that increasing the number of modes is the most computationally costly action we can take, in theory as well as practice, and should be minimised. For the rest of the results presented, we have used $20$ modes for each end, $1000$ eigenvalues and a mesh refinement of $30.$

\subsection{Observing a resonance transition from one sheet to another}
\label{newsheet}
This subsection will present the first explicit calculation of the scattering matrix beyond the first non-physical sheet $J= \lbrace 1 \rbrace$ of $Z$ ($Z$ was defined by equation (\ref{jay}) in section~\ref{sheetsstuff}). We will be looking at the same domain as we did in~\ref{levitinsrectangle}, with the same conventions; an infinite rectangle in $\mathbb{R}^2$ of width $2$ with a single, circular `obstruction' of radius $R$ who's centre is moved  up by $\delta$ from the middle of the waveguide. In~\ref{levitinsrectangle} we used the parameterisation $\lambda \mapsto \lambda^2 $ to enable the comparison of our results with other peoples', from now on we will no longer do this. Of particular interest here is the case where $R=0.2.$ As $\delta$ increases from $0.6$ to $0.7$ the resonance moves from the sheet $J=\lbrace 1 \rbrace$ to $J= \lbrace 2 \rbrace$ as can be seen in Table~\ref{restranstable}. This was the first circumstance where such a phenomena has been observed.

\begin{figure}[H]
\centering
 \begin{tabular}{| c  ||c| c  |   }
\hline
$R=0.2 $   & $J=\lbrace 1 \rbrace$ &$J=\lbrace 2 \rbrace$ 
\\ 
\hhline{|=||=|=|}
$\delta=0$  & $2.403\light{6}$&- \\  \hline
$\delta=0.1$  &$2.407\light{12} +0.00006 i$ &-   \\  \hline
$\delta=0.2$  &$2.417\light{09} +0.0002\light{1} i$ & - \\  \hline 
$\delta=0.3$  & $2.431\light{70} +0.0004\light{0}i$& - \\  \hline 
$\delta=0.4$  &$2.447\light{77}+0.0005\light{9}i$ &-  \\  \hline
$\delta=0.5$  &$2.461\light{01} +0.0005\light{3}i$ &-  \\  \hline 
$\delta=0.6$  &$2.467\light{25} +0.0001\light{3} i$ &- \\  \hline 
$\delta=0.7$  &- & $2.464\light{75}+0.000\light{63}i$\\  \hline   
   \end{tabular}
   \captionof{table}{Resonances for the domain $R=0.2$ showing the resonance moving to a different sheet of $Z$ as $\delta$ increases from $0.6$ to $0.7$}
   \label{restranstable}
\end{figure}

We will also include some contour plots of the absolute value of the determinant of the scattering matrix for this occurrence. We can observe that the ``tail'' of the resonance is visible on both sheets prior to the resonance crossing over.

\begin{figure}[H]
\centering
\begin{tabularx}{\textwidth}{XX}
\includegraphics[clip, width=6cm, height =4cm]{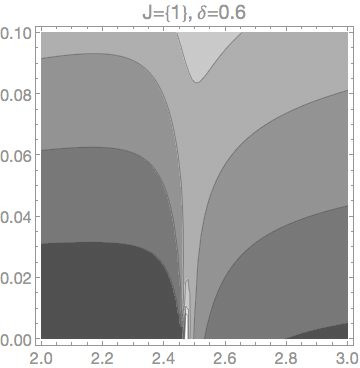}&\includegraphics[clip, width=6cm, height =4cm]{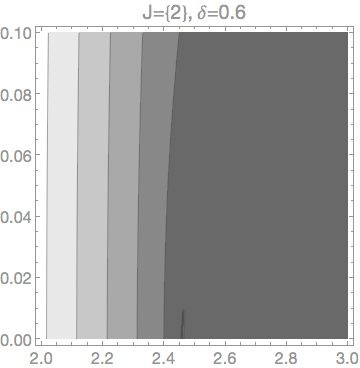}\\
\includegraphics[clip, width=6cm, height =4cm]{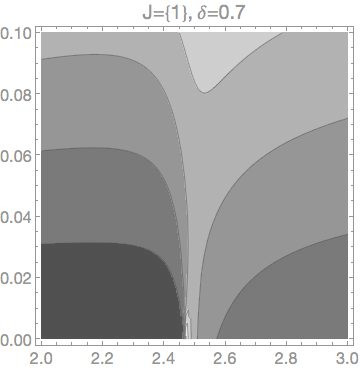}&\includegraphics[clip, width=6cm, height =4cm]{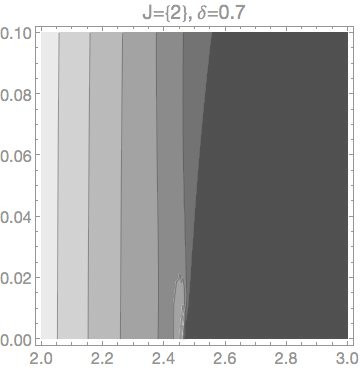}\\
\end{tabularx}
\caption{Contour plots of the absolute value of the the determinant of $S_J(\lambda).$ The resonance can be seen crossing from sheet $ J=\lbrace 1 \rbrace$ to $J=\lbrace 2 \rbrace$ as $\delta$ is increased from $0.6$ to $0.7$. $R=0.2$ throughout.}
\end{figure}

\subsection{Adding ``Ends'' to a Circle}
\label{circends}

In this example, we have added a single, infinite rectangular `end'  to a circle of radius $2$ and the width of the `end' has been varied. We have searched for resonances, on sheets $J=\lbrace1\rbrace, \ J=\lbrace2\rbrace$ and $J=\lbrace3\rbrace,$ within in the search area given by:

 \begin{equation}
 \label{searcharea}
 \lbrace \lambda: 0 \leq \mathrm{Re}(\lambda) \leq 15,-3 \leq \mathrm{Im}(\lambda)\leq 3 \rbrace,
 \end{equation}
 and recorded the results the results found in Table~\ref{circleend2}.  The choice of sheets and and search area on which to look for resonances was somewhat arbitrary, though we believe they provide a wide enough `field of vision' to observe interesting phenomena, but are limited enough to minimise compute time and cost. We have chosen this as an example because we feel it interesting to observe that when the width of the end is close to zero, the resonances are very close to the Neumann eigenvalues of the circle of radius $2,$ of which we have given the first  $9$ non-zero examples in Table~\ref{circradev}. When the width of the end is continuously increased, they migrate away from these values in distinct paths.
We are confident of the accuracy of the resonances calculated to least three decimal places, though we have included the fourth place in a lighter shade for the reader's information. In addition to this, we have plotted them on graphs in Figure~\ref{circleendplots} with colour coded markers indicating the respective sheet of $Z$ they reside on; black for $J=\lbrace1 \rbrace,$ red for $J=\lbrace2 \rbrace$ and green for $J=\lbrace3 \rbrace.$ In the case of varying widths, the paths taken by the resonances as the widths increase are clearly visible, and we have included them.
  \begin{figure}[H]

\begin{tabular}{|c| c|c|c|c|c|c|c| c|}
\hline $0.8476$ &  $2.3323$ & $ 3.6709$& $4.4130$& $7.0698$&$7.1068$&$10.2911$&$11.2442$&$12.3059$\\ 
     \hline
   \end{tabular}
        \captionof{table}{Neumann eigenvalues for a circle of radius $2.$}
  \label{circradev}
      \end{figure}

\begin{figure}[H]
\centering
 \begin{tabular}{ccc   }
\includegraphics[trim = 10mm 10mm 20mm 170mm, clip, width=3.5cm]{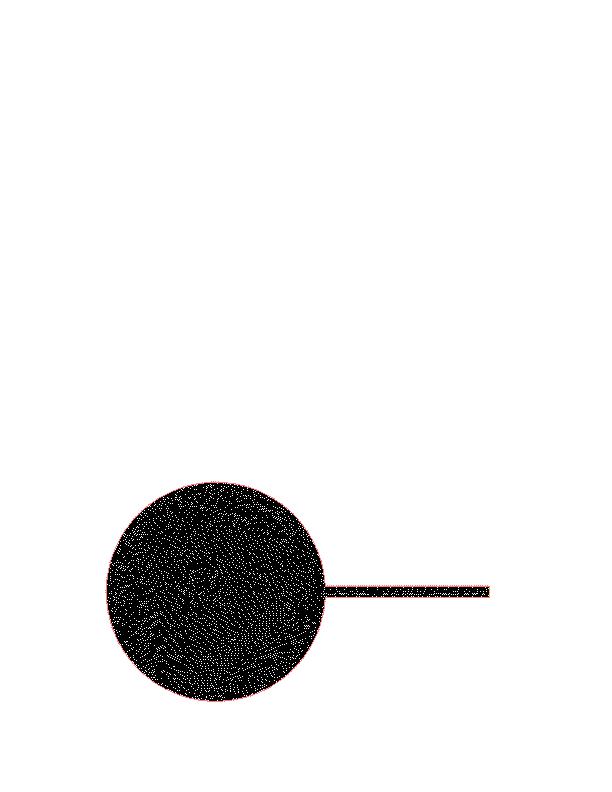}&\includegraphics[trim = 10mm 10mm 20mm 170mm, clip, width=3.5cm]{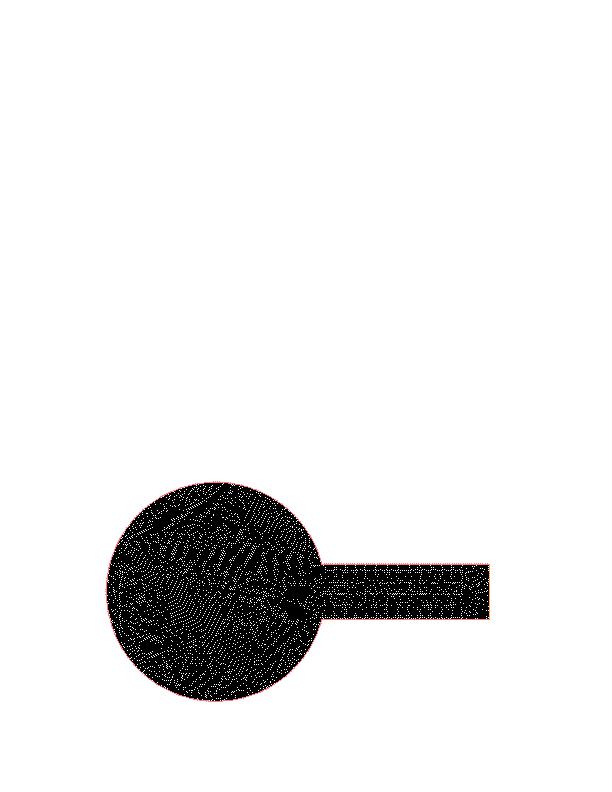}&\includegraphics[trim = 10mm 10mm 20mm 170mm, clip, width=3.5cm]{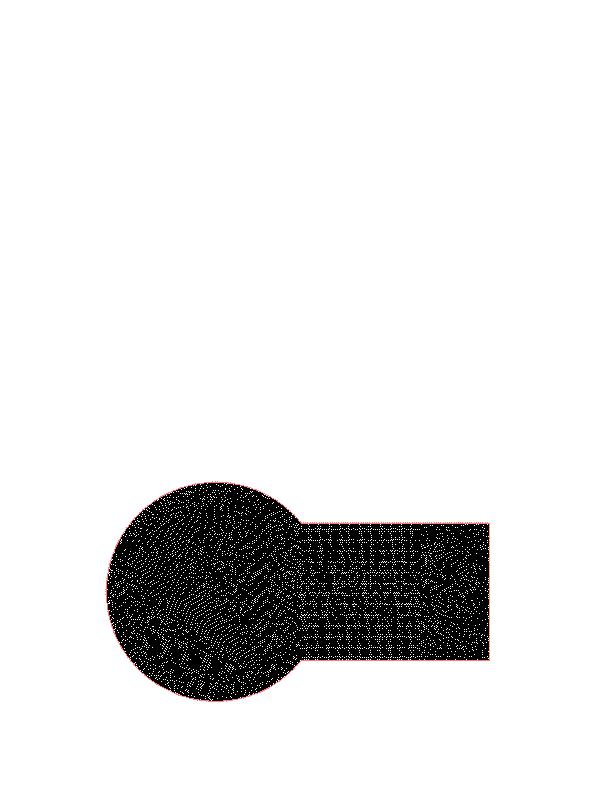}
\end{tabular}
\caption{An example of some internal domains produced by FreeFEM++ for these calculations.}
\label{circleendgraphs}
\end{figure}

\begin{figure}[H]
\begin{tabularx}{\textwidth}{| X | X | X | X |}
\hhline{|====|}
\multicolumn{4}{|c|}{$J=\lbrace1 \rbrace$}\\
  \hline
$w=0.1$& $w=0.2$ &$w=0.5$  &$w=1$\\
\hline 
$0.849\light{6} +0.020\light{6} i$&$0.854\light{1} +0.040\light{8} i$&$0.875\light{3} +0.099\light{2} i$&$0.925\light{5} +0.195\light{7} i$\\
$2.338\light{8} +0.041\light{6} i$&$2.352\light{5} +0.080\light{7} i$&$2.413 +0.184\light{8} i$       &$2.543\light{1} +0.333\light{4} i$\\
$3.674\light{2} +0.014\light{5} i$&$3.681\light{4} +0.026\light{5} i$&$3.709\light{1} +0.051\light{2} I $&$3.757\light{7} +0.080\light{2} i$\\
$4.424\light{6} +0.066\light{4} i$&$4.448\light{0} +0.128\light{5} i$&$4.545\light{2} +0.300\light{9} I $&$4.742\light{5} +0.576\light{4} i$\\
$7.094\light{7} +0.002\light{2} i$&$7.094\light{6} +0.001\light{1} i$&$7.094\light{7} +0.000\light{4} i$&$7.857\light{3} +1.057\light{3} i$\\
$7.112\light{4} +0.132\light{8} i$&$7.172\light{1} +0.254\light{1} i$&$7.406\light{5} +0.565\light{7} i$&$10.772\light{6} +0.169\light{9} i$ \\
$10.329\light{3} +0.118\light{6} i$&$10.403\light{7} +0.207\light{1} i$&$10.666\light{4} +0.294\light{0} i$&$12.288\light{4} +0.894\light{5} i$\\
$11.254\light{7} +0.057\light{2} i$&$11.273\light{3}+0.115\light{1} i$&$11.368 +0.354\light{67} i$&$12.099\light{1} +0.352\light{1} i$\\
$12.312\light{2} +0.026\light{3} i$&$12.323\light{7} +0.050\light{2} i$&$12.356\light{7} +0.134\light{8} i$&$14.720\light{5} +0.665\light{4} i$\\
$14.113\light{9} +0.152\light{7} i$&$14.194\light{9} +0.276\light{6} i$&$14.464\light{9} +0.534\light{5} i$&\\
          \hline
          
           $ w=1.5$& $w=2$ &$w=2.5$  &$w=3$\\
   \hline
$0.990\light{0} +0.301\light{0 }i$&$1.071\light{0} +0.426\light{4 }i$&$1.175\light{2} +0.589\light{8 }i$&$1.317\light{6} +0.829\light{0 }i$\\
$2.696\light{5 }+0.462\light{7 }i$&$2.874\light{6} +0.561\light{4 }i$&$3.069\light{6} +0.576\light{7 }i$&$3.171\light{3} +0.447\light{6 }i$\\
$3.815\light{4} +0.122\light{4 }i$&$3.900\light{0} +0.205\light{4 }i$&$4.056\light{9} +0.375\light{0 }i$&$4.505\light{5} +0.548\light{6 }i$\\
$4.940\light{3} +0.859\light{7 }i$&$5.097\light{0} +1.141\light{7 }i$&$5.184\light{3} +1.349\light{1 }i$&$5.367\light{6} +1.402\light{9 }i$\\
$7.127\light{6} +0.001\light{4 }i$&$7.175\light{7} +0.015\light{1 }i$&$9.252\light{8} +2.425\light{7 }i$&\\
$8.293\light{3} +1.616\light{5 }i$&$8.610\light{2} +2.138\light{0 }i$&&\\
$10.723\light{4} +0.088\light{3 }i$&$11.929\light{0} +0.075\light{8 }i$&&\\
$12.012\light{9} +0.189\light{7 }i$&$13.78\light{0} +2.663\light{5 }i$&&\\
$13.145\light{9} +2.049\light{3 }i$&&&\\
$14.525\light{2} +0.336\light{8}i$&&&\\
      \hline \end{tabularx}

\end{figure}
\raggedbottom

\begin{figure}[H]
\begin{tabularx}{\textwidth}{| X | X | X | X |}
\hhline{|====|}
          \multicolumn{4}{|c|}{$J=\lbrace2 \rbrace$}\\
          \hline
         $w=0.1$& $w=0.2$ &$w=0.5$  &$w=1$\\
      \hline 
&&&$10.253\light{6} +0.115\light{62} i$\\
&&&$11.243\light{5} +0.017\light{77} i$\\
&&&$14.216\light{7} +0.427\light{36} i$\\
\hline
         $ w=1.5$& $w=2$ &$w=2.5$  &$w=3$\\
      \hline
$7.243\light{9} +0.469\light{8} i$&$4.606\light{5} +0.482\light{9} i$&$2.369\light{5} +0.338\light{9} i$&$2.604\light{2} +0.651\light{1} i$\\
$7.743\light{3} +0.929\light{9} i$&$11.076\light{9} +0.0884\light{2 }i$&$4.986\light{9} +0.832\light{4} i$&$5.521\light{7} +1.239\light{3} i$\\
$10.902\light{2} +0.727\light{5} i$&$11.826\light{4} +1.3861\light{6} i$&$7.115\light{6} +0.001\light{1} i$&$7.135\light{6} +0.019\light{2}i$\\
$11.197\light{0} +0.123\light{9} i$	&&$8.367\light{0} +1.4419\light{8 }i$&$9.109\light{6} +1.949\light{2 }i$\\
&&$10.996\light{8} +0.035\light{5} i$&$14.161\light{6} +2.754\light{5 }i$\\
&&$12.763\light{1} +2.163\light{4} i$&\\

\hhline{|====|}
  \multicolumn{4}{|c|}{$J=\lbrace3 \rbrace$}\\
          \hline
              $w=0.1$& $w=0.2$ &$w=0.5$  &$w=1$\\
      \hline 
      &&&\\
        \hline
          $ w=1.5$& $w=2$ &$w=2.5$  &$w=3$\\
      \hline 
&$10.367\light{3} +0.272\light{1}i$&$7.063\light{9} +0.254\light{6} i$&$7.630\light{6} +0.714\light{7} i$\\
&$14.786\light{7} +0.868\light{2}i$&$11.83\light{1} +0.087\light{3} i$&$11.626\light{8} +0.098\light{8} i$\\
&&$11.16\light{1} +1.083\light{3}i$&$12.405\light{9} +2.043\light{0} i$\\
  \hline \end{tabularx}
\captionof{table}{Resonances calculated when a singe rectangular `end' is added to a circle of radius $2$ in two dimensions when $ \lambda $ is in the sheets $J=\lbrace 1 \rbrace, \ J=\lbrace 2 \rbrace$ and $J=\lbrace 3 \rbrace$ of $Z$ as indicated. $\omega$ denotes the width of the end which increases from $0.1$ to $3$.}
\label{circleend2}
\end{figure}

\begin{figure}[H]
\begin{tabularx}{\textwidth}{XX}
\includegraphics[clip, width=6cm, height =4cm]{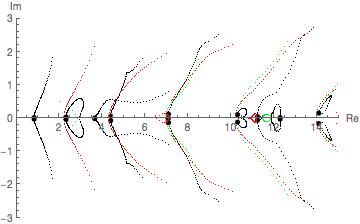}\caption*{$w=0.1$}&\includegraphics[clip, width=6cm, height =4cm]{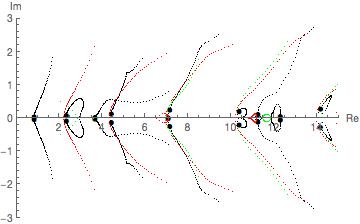}\caption*{$w=0.2$}\\
\includegraphics[clip, width=6cm, height =4cm]{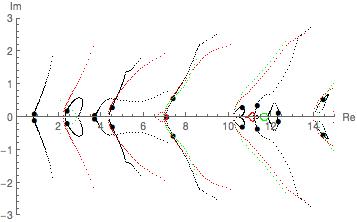}\caption*{$w=0.5$}&\includegraphics[clip, width=6cm, height =4cm]{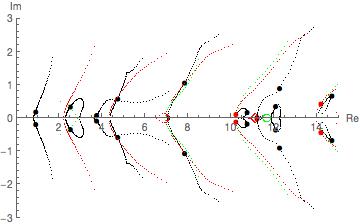}\caption*{$w=1.0$}\\

\includegraphics[clip, width=6cm, height =4cm]{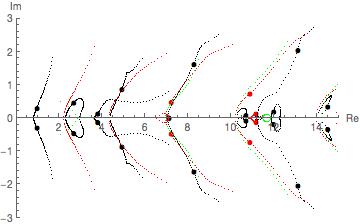}\caption*{$w=1.5$}&\includegraphics[clip, width=6cm, height =4cm]{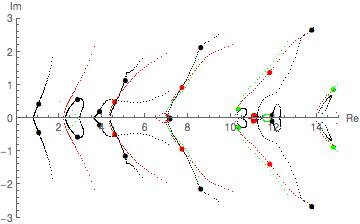}\caption*{$w=2.0$}\\
\includegraphics[clip, width=6cm, height =4cm]{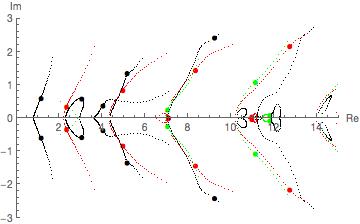}\caption*{$w=2.5$}&\includegraphics[clip, width=6cm, height =4cm]{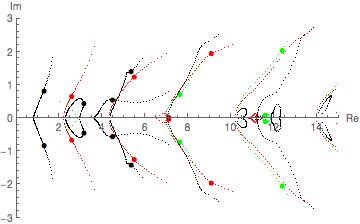}\caption*{$w=3.0$}\\
\end{tabularx}
\caption{Some colour-coded plots of the location of resonances. In this instance the width of a single end is changed. The larger markers are the resonances for the domain indicated, the smaller markers are the entire family of domains with varying widths. This allows the reader to visualise the paths the resonances take as the width of the end increases. Similar, but not identical, paths can be observed with higher numbers of ends.}
\label{circleendplots}
\end{figure}  

\subsection{Animated motion of resonances as the domain is continuously altered}

Due to the speed of the method it has been possible to animate the motion of the resonances as the domains are continuously changed. An animation of the results found in~\ref{circends} can be found in the link below along with many more. Of particular interest to the reader are occurrences where resonances cross sheets, higher order resonances bifurcate, resonances loop etc. The same colour coding convention applies in the animations as in Figure~\ref{circleendplots}. Many other similar animations have been produced
\\

\begin{itemize}
\item Varying the width of one rectangular end \\
\url{https://www.youtube.com/watch?v=6AryY-dHRAM }
\item Varying the width of two equispaced rectangular ends \\
\url{https://www.youtube.com/watch?v=GRl3vY2IkOY}
\item Varying the size of a centrally placed circular `obstruction', one end of width $1$ \\
\url{https://www.youtube.com/watch?v=T6Xap-VQHmQ}
\item Varying the size of a centrally placed circular `obstruction', two equispaced ends of width $1$ \\
\url{https://www.youtube.com/watch?v=WhcdPBotn9A}
\item Varying the size of a centrally placed circular `obstruction', three equispaced ends of width $1$ \\
\url{https://www.youtube.com/watch?v=UP4P9xW7ATI}

\item Varying the angle between two rectangular ends of width $1$ \\
\url{https://www.youtube.com/watch?v=8HDXRFmqxfU}

\item Moving a circular `obstruction' of radius $0.5$ in the x axis, 3 equally spaced ends of width $1.$
\url{https://www.youtube.com/watch?v=B8WmXaDo0bs}
\item Moving a circular `obstruction' of radius $0.5$ in the y axis, 3 equally spaced ends of width $1.$ \url{https://www.youtube.com/watch?v=WncB1xoSSzs}
\end{itemize}

If the reader has any difficulty in accessing these videos, then they can be provided, on request, by other means.
\subsection{Time delay and scattering length}
\label{tdsl}
The notions of time delay and scattering length hail from dynamic scattering theory; they can both be calculated using the scattering matrix. Reed and Simon mention time delay briefly in their text~\cite{rs3}. M\"{u}ller and Strohmaier have also covered time delay and scattering length in their paper \cite{muller}, where they give results that relate the time delay to the geometry of the internal domain. In this section, we will apply this to a pair of the domains featured above. It should be noted that the $\lambda$ in this context will be a real number less than $\mu_1,$ representing the energy of the system and not an element of the $Z.$ The Appendix of M\"{u}ller and Strohmaier's paper provides an overview of the time delay in this setting \cite{mullst}. We will take the (non standard) definition to be: 
\begin{definition}\textbf{Time delay}
$$T(\lambda)=-2\sqrt{\lambda}S^{-1}(\lambda).S'(\lambda)$$
when $\lambda=0,$ we define this to be the \textbf{scattering length}.
\end{definition}
Wigner and Eisenbud were the first to present the time delay in this manner for potential scattering and $T(\lambda)$ is often called the Eisenbud-Wigner time delay operator~\cite{wigner}\cite{eisenbud}. 
M\"{u}ller and Strohmaier have, amongst other things, proved this formula for the case of manifolds with cylindrical ends and, in the case of a single end 

$$T(0)=2\frac{\text{Vol}(X)}{\text{Vol}(\Gamma)}.$$
We will pick a selection of our single ended domains from above and plot their time delay as $\lambda$ approaches $0.$

\begin{figure}[H]
\centering
\begin{tabularx}{\textwidth}{XX}
\includegraphics[clip, width=7cm, height =5cm]{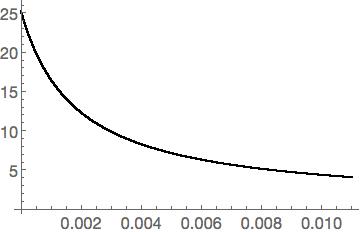}\caption{Circle of radius $2,$ with end width $1.5$. $T(0)$ should theoretically be $25.1327.$ \ignore{25.0877}}&\includegraphics[clip, width=7cm, height =5cm]{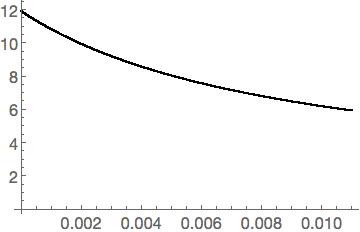}\caption{Circle of radius $2$ with end width $1,$ obstacle radius $0.5.\ T(0)$ \ignore{11.7719} should theoretically be $11.781.$}

\end{tabularx}
\end{figure}

\subsection{Limitations and future directions}
We will finish by informing the reader of the limitations of what we have done as things stand and, give some suggestions for further work in this immediate area.
\begin{itemize}
\item The algorithm is currently limited to Neumann or Acoustic waveguides in Euclidian space; it can be extended, with minimal additional theoretical work, to those with Dirichlet or mixed boundary conditions, as well as those with non-Euclidian metrics on either the internal domain or the ends. There is also the possibility to look at potentials supported on the internal domain.
\item Applying this method to higher dimensions and more complicated domains is possible, though the computational cost will rise very significantly and greater care needs to be taken with accuracy. 
\item Continuing to use the current method as it is to either search for interesting phenomena with respect to the distribution of resonances as David Borthwick has done for hyperbolic surfaces~\cite{db}, or perform domain optimisation and create domains from givern patterns of resonances.
\item Adopting the boundary element method~\cite{timobem} in place of the finite element method or adopt finite element method techniques that minimise dispersion error and try to improve accuracy yet further~\cite{dispersion}.
\end{itemize}
\section{References}

\bibliographystyle{plain}
\bibliography{paper_revised3}

\begin{thebibliography}{10}

\bibitem{alexnew}
M.~Levitin and A.~Strohmaier, ``Computations of eigenvalues and resonances on
  perturbed hyperbolic surfaces with cusps,'' 2019.

\bibitem{rm}
R.~Melrose, {\em Geometric Scattering Theory}.
\newblock Stanford Lectures: Distinguished Visiting Lecturers in Mathematics,
  Cambridge University Press, 1995.

\bibitem{lm}
M.~Levitin and M.~Marletta, ``A simple method of calculating eigenvalues and
  resonances in domains with infinite regular ends,'' {\em Proceedings of the
  Royal Society of Edinburgh: Section A Mathematics}, vol.~138, pp.~1043--1065,
  10 2008.

\bibitem{apv}
A.~Aslanyan, L.~Parnovski, and D.~Vassiliev, ``Complex resonances in acoustic
  waveguides,'' {\em The Quarterly Journal of Mechanics and Applied
  Mathematics}, vol.~53, no.~3, pp.~429--447, 2000.

\bibitem{Jiang_2020}
T.~Jiang and Y.~Xiang, ``Perfectly-matched-layer method for optical modes in
  dielectric cavities,'' {\em Physical Review A}, vol.~102, Nov 2020.

\bibitem{christiensen}
T.~Christiansen, ``Some upper bounds on the number of resonances for manifolds
  with infinite cylindrical ends,'' {\em Annales Henri Poincare}, vol.~3,
  no.~5, pp.~895--920, 2002.

\bibitem{wigner}
E.~P. Wigner, ``Lower limit for the energy derivative of the scattering phase
  shift,'' {\em Phys. Rev.}, vol.~98, pp.~145--147, Apr 1955.

\bibitem{eisenbud}
L.~Eisenbud, {\em Dissertation, Unpublished}.
\newblock PhD thesis, Princeton University, 1948.

\bibitem{trappedmodes}
D.~V. Evans, M.~Levitin, and D.~Vassiliev, ``Existence theorems for trapped
  modes,'' {\em Journal of Fluid Mechanics}, vol.~261, pp.~21--31, 2 1994.

\bibitem{trappedmodes2}
E.~B. Davies and L.~Parnovski, ``Trapped modes in acoustic waveguides,'' {\em
  Quart. J. Mech. Appl. Math.}, vol.~51, no.~3, pp.~477--492, 1998.

\bibitem{hawkparnovski}
H.~Hawkins and L.~Parnovski, ``Trapped modes in a waveguide with a thick
  obstacle,'' {\em Mathematika}, vol.~51, no.~1-2, pp.~171--186 (2005), 2004.

\bibitem{pencilev}
E.~R. Johnson, M.~Levitin, and L.~Parnovski, ``Existence of eigenvalues of a
  linear operator pencil in a curved waveguide---localized shelf waves on a
  curved coast,'' {\em SIAM J. Math. Anal.}, vol.~37, no.~5, pp.~1465--1481
  (electronic), 2006.

\bibitem{rs1}
M.~Reed and B.~Simon, {\em I: Functional Analysis}.
\newblock Methods of Modern Mathematical Physics, Elsevier Science, 1981.

\bibitem{tetgen}
G.~Roddick", ``{Computation of scattering matrices and resonances for
  waveguides},'' "6" "2016".

\bibitem{french}
L.~Guillop{\'e}, ``Th\'eorie spectrale de quelques vari\'et\'es \`a bouts,''
  {\em Ann. Sci. \'Ecole Norm. Sup. (4)}, vol.~22, no.~1, pp.~137--160, 1989.

\bibitem{rs3}
M.~Reed and B.~Simon, {\em Methods of modern mathematical physics. {III}}.
\newblock Academic Press [Harcourt Brace Jovanovich, Publishers], New
  York-London, 1979.
\newblock Scattering theory.

\bibitem{ip2}
A.~Girouard and I.~Polterovich, ``{Spectral Geometry of the Steklov Problem}.''
  {2014}.

\bibitem{mt}
M.~Taylor, {\em Partial Differential Equations I: Basic Theory}.
\newblock Applied Functional Analysis: Applications to Mathematical Physics,
  Springer, 1996.

\bibitem{hormander3}
F.~L. Gall, ``Faster algorithms for rectangular matrix multiplication,'' {\em
  CoRR}, vol.~abs/1204.1111, 2012.

\bibitem{steve_johnson_casimir}
``Casimir physics,'' {\em Lecture Notes in Physics}, 2011.

\bibitem{timobem}
W.~undefinedmigaj, T.~Betcke, S.~Arridge, J.~Phillips, and M.~Schweiger,
  ``Solving boundary integral problems with bem++,'' {\em ACM Trans. Math.
  Softw.}, vol.~41, Feb. 2015.

\bibitem{ab_mps}
A.~Barnett, ``Mpspack : 2d helmholtz scattering and eigenvalue problems via
  particular solutions and integral equations,'' 2016.

\bibitem{melrosepolybounds}
R.~B. Melrose, ``Polynomial bound on the distribution of poles in scattering by
  an obstacle,'' {\em Journaes Equations aux derivees partielles},
  pp.~1--8, 1984.

\bibitem{zworskipolybounds}
M.~Zworski, ``Sharp polynomial bounds on the number of scattering poles of
  radial potentials,'' {\em Journal of Functional Analysis}, vol.~82, no.~2,
  pp.~370 -- 403, 1989.

\bibitem{tanyaasymp}
Z.~M. Christiansen, Tanya, ``Spectral asymptotics for manifolds with
  cylindrical ends,'' {\em Annales de l'institut Fourier}, vol.~45, no.~1,
  pp.~251--263, 1995.

\bibitem{zworskig}
L.~Guillope and M.~Zworski, ``Scattering asymptotics for riemann surfaces,''
  {\em Annals of Mathematics}, vol.~145, no.~3, pp.~pp. 597--660, 1997.

\bibitem{aslanyanebdcalc}
A.~M. Aslanyan and E.~B. Davies, ``{Separation of variables in perturbed
  cylinders},'' Tech. Rep. math.SP/0012113, Dec 2000.

\bibitem{db}
D.~Borthwick, ``Distribution of resonances for hyperbolic surfaces,'' {\em Exp.
  Math.}, vol.~23, no.~1, pp.~25--45, 2014.

\bibitem{betcke1}
L.~N. Trefethen and T.~Betcke, ``Computed eigenmodes of planar regions,'' in
  {\em Recent advances in differential equations and mathematical physics},
  vol.~412 of {\em Contemp. Math.}, pp.~297--314, Amer. Math. Soc., Providence,
  RI, 2006.

\bibitem{betcke2}
T.~Betcke and L.~N. Trefethen, ``Reviving the method of particular solutions,''
  {\em SIAM Rev.}, vol.~47, no.~3, pp.~469--491 (electronic), 2005.

\bibitem{nazarov1}
S.~A. Nazarov, K.~Ruotsalainen, and P.~Uusitalo, ``Bound states of waveguides
  with two right-angled bends,'' {\em J. Math. Phys.}, vol.~56, no.~2,
  pp.~021505, 24, 2015.

\bibitem{freefem}
F.~Hecht, ``New development in freefem++,'' {\em J. Numer. Math.}, vol.~20,
  no.~3-4, pp.~251--265, 2012.

\bibitem{arpack}
R.~B. Lehoucq, D.~C. Sorensen, and C.~Yang, ``Arpack users guide: Solution of
  large scale eigenvalue problems by implicitly restarted arnoldi methods.,''
  1997.

\bibitem{arnoldi}
D.~C. Sorensen, ``Implicitly restarted arnoldi/lanczos methods for large scale
  eigenvalue calculations,'' 1996.

\bibitem{betke_mps}
T.~Betcke and L.~Trefethen, ``Reviving the method of particular solutions,''
  {\em SIAM Review}, vol.~47, pp.~469--491, 09 2005.

\bibitem{muller}
W.~M\"{u}ller, ``On the analytic continuation of rank one eisenstein series,''
  {\em Geometric \& Functional Analysis GAFA}, vol.~6, no.~3, pp.~572--586,
  1996.

\bibitem{mullst}
W.~M\"{u}ller and A.~Strohmaier, ``Scattering at low energies on manifolds with
  cylindrical ends and stable systoles,'' {\em Geometric and Functional
  Analysis}, vol.~20, no.~3, pp.~741--778, 2010.

\bibitem{dispersion}
Z.~C. He, A.~G. Cheng, G.~Y. Zhang, Z.~H. Zhong, and G.~R. Liu, ``Dispersion
  error reduction for acoustic problems using the edge-based smoothed finite
  element method (es-fem),'' {\em International Journal for Numerical Methods
  in Engineering}, vol.~86, no.~11, pp.~1322--1338, 2011.

\end{thebibliography}

\end{document}